\documentclass[a4paper,11pt]{article}

\usepackage[english]{babel}
\usepackage{cite}
\usepackage{amsfonts}
\usepackage{latexsym}
\usepackage{amsthm}
\usepackage{amsopn}
\usepackage{amsmath}
\usepackage[all]{xy}
\usepackage{verbatim}
\usepackage{amssymb}
\usepackage{mathrsfs}
\usepackage{float}
\usepackage[active]{srcltx}
\usepackage{fullpage}
\usepackage{hyperref}
\usepackage{manyfoot}
\usepackage{amscd}
\usepackage[dvips]{graphics}
\usepackage[dvips]{graphicx}
\usepackage{amscd}

\theoremstyle{remark}
\newtheorem{remark}{Remark}

\theoremstyle{rmk}
\newtheorem{rmk}[remark]{Remark}

\theoremstyle{plain}
\newtheorem{theorem}[remark]{Theorem}
\newtheorem{proposition}[remark]{Proposition}

\newtheorem{corollary}[remark]{Corollary}

\newtheorem*{theoA}{Theorem A}
\newtheorem*{theoB}{Theorem B}

\theoremstyle{definition}
\newtheorem{definition}[remark]{Definition}

\newcommand{\mQ}{\mathcal{Q}}

\newcommand{\oo}{\mathcal{O}}

\newcommand{\mL}{\mathcal{L}}

\newcommand{\ZZ}{\mathbb{Z}}

\newcommand{\FF}{\mathcal{F}} 

\newcommand{\mLd}{\mathcal{L}_{\delta}}
\newcommand{\wA}{\widehat{A}}

\newcommand{\lr}{\longrightarrow}

\renewcommand{\to}{\longrightarrow}

\title{Surfaces with $p_g=q=2$, $K^2=6$ and Albanese map \\ of degree $2$}

\author{Matteo Penegini, Francesco Polizzi}

\date{}

\begin{document}

\maketitle

\begin{abstract}
We classify minimal surfaces of general type with $p_g=q=2$ and
$K^2=6$ whose Albanese map is a generically finite double cover.
We show that the corresponding moduli space is the disjoint union
of three generically smooth, irreducible components
$\mathcal{M}_{Ia}$, $\mathcal{M}_{Ib}$, $\mathcal{M}_{II}$ of
dimension $4$, $4$, $3$, respectively.
\end{abstract}

\Footnotetext{{}}{\textit{2010 Mathematics Subject
Classification}: 14J29, 14J10}

\Footnotetext{{}} {\textit{Keywords}: Surface of general type,
 abelian surface, Albanese map}


\section{Introduction}

Minimal surfaces $S$ of general type with $p_g=q=2$ fall into two
classes according to the behavior of their Albanese map $\alpha
\colon S \to A$. Indeed, since $q=2$, either $\alpha(S)=C$, where
$C$ is a smooth curve of genus $2$, or $\alpha$ is surjective and
$S$ is of \emph{Albanese general type}.

The surfaces which belong to the former case satisfy $K^2_S=8$ and
are nowadays completely classified, see \cite{Z03} and
\cite{Pe09}. Those belonging  to the latter case present a much
richer and subtler geometry, and their full description is still
missing; we refer the reader to the introduction of \cite{PP10}
and the references given there for a recent account on this topic.

So far, the only known example of a surface of general type with
$p_g=q=2$ and $K^2_S=6$ was the one given in \cite{Pe09}; in that
case, the Albanese map is a generically finite quadruple cover of an
abelian surface with a polarization of type  $(1, \, 3)$.

As the title suggests, in this paper we investigate surfaces with
the above invariants and whose Albanese map is a generically finite
\emph{double} cover. The results that we obtain are quite
satisfactory, indeed we are not only able to show the existence of
such \emph{new} surfaces, but we also provide their complete classification,
together with a detailed description of their moduli space.

Before stating our results, let us introduce some notation. Let
$(A, \,\mathcal{L})$ be a $(1, \,2)$-polarized abelian surface and
let us denote by $\phi_2 \colon A[2] \to \widehat{A}[2]$ the
restriction of the canonical homomorphism
$\phi_{\mathcal{L}} \colon A \to \widehat{A}$ to the subgroup of
$2$-division points. Then $\textrm{im}\, \phi_2$ consists of four
line bundles $\{\mathcal{O}_A, \, \mathcal{Q}_1, \, \mathcal{Q}_2,
\, \mathcal{Q}_3\}$. Let us denote by $\textrm{im}\,
\phi_2^{\,\times}$ the set $\{\mathcal{Q}_1, \, \mathcal{Q}_2, \,
\mathcal{Q}_3\}$.

Our first result is

\begin{theoA}[see Theorem \ref{th.A}] \label{theo.A}
Given an  abelian surface $A$ with a symmetric polarization $\mL$
of type $(1, \,2)$, not of product type,  for any $\mQ \in
\emph{im}\, \phi_2$ there exists a curve $D \in |\mL^2 \otimes
\mQ|$ whose unique non-negligible singularity is an ordinary
quadruple point at the origin $o \in A$. Let $\mQ^{1/2}$ be a
square root of $\mQ$, and if $\mQ=\oo_A$ assume moreover
$\mQ^{1/2} \neq \oo_A$. Then the minimal desingularization $S$ of
the double cover of $A$ branched over $D$ and defined by $\mL
\otimes \mQ^{1/2}$ is a minimal surface of general type with
$p_g=q=2$, $K_S^2=6$ and Albanese map of degree $2$.

Conversely, every minimal surface of general type with $p_g=q=2$,
$K_S^2=6$ and Albanese map of degree $2$ can be constructed in
this way.
\end{theoA}

When $\mQ=\mQ^{1/2}=\oo_A$ we obtain instead a minimal surface with
$p_g=q=3$, see Proposition \ref{prop.invariants} and Remark
\ref{rem.pg=3}.

We use the following terminology:
\begin{itemize}
\item if $\mQ =\oo_A$ we say that $S$ is a \emph{surface of type}
$I$. Furthermore, if $\mathcal{Q}^{1/2} \notin \textrm{im} \,
\phi_2$ we say that $S$ is of type $Ia$, whereas if
$\mathcal{Q}^{1/2} \in \textrm{im} \, \phi_2^{\,\times}$  we say
that $S$ is of type $Ib;$ \item if $\mQ \in
\textrm{im}\,\phi_2^{\, \times}$ we say that $S$ is a
\emph{surface of type} $II$.
\end{itemize}

Since $q=2$, the results in \cite{Ca91} imply that the degree of the
Albanese map is a topological invariant, see Proposition
\ref{prop.degree.alb}. Therefore we may consider the moduli space
$\mathcal{M}$ of minimal surfaces of general type with $p_g=q=2$,
$K_S^2=6$ and Albanese map of degree $2$. Let $\mathcal{M}_{Ia}$,
$\mathcal{M}_{Ib}$, $\mathcal{M}_{II}$ be the algebraic subsets
whose points parameterize isomorphism classes of surfaces of type
$Ia$, $Ib$, $II$, respectively. Therefore $\mathcal{M}$ can be
written as the disjoint union
\begin{equation*}
\mathcal{M}=\mathcal{M}_{Ia} \sqcup \mathcal{M}_{Ib} \sqcup
\mathcal{M}_{II}.
\end{equation*}

Our second result is

\begin{theoB}[see Theorem \ref{thm.B}] \label{theo.B} The following holds:
\begin{itemize}
\item[$\boldsymbol{(i)}$] $\mathcal{M}_{Ia}$, $\mathcal{M}_{Ib}$,
$\mathcal{M}_{II}$ are the connected components of $\mathcal{M};$
\item[$\boldsymbol{(ii)}$]
these are also \emph{irreducible} components of the moduli space of
minimal surfaces of general type$;$
\item[$\boldsymbol{(iii)}$]
$\mathcal{M}_{Ia}, \, \mathcal{M}_{Ib}, \, \mathcal{M}_{II}$ are
generically smooth, of dimension $4, \, 4, \, 3$, respectively$;$
\item[$\boldsymbol{(iv)}$]
the general surface in $\mathcal{M}_{Ia}$ and $\mathcal{M}_{Ib}$ has
ample canonical class$;$ all surfaces in $\mathcal{M}_{II}$ have
ample canonical class.
\end{itemize}
\end{theoB}


This work is organized as follows.

In Section \ref{sec.prel} we fix notation and terminology and we
prove some technical results on abelian surfaces with
$(1,2)$-polarization which are needed in the sequel of the paper.

 In the Section \ref{sec.d=2} we
give  the proof of Theorem A. Moreover we provide a description of
the canonical system $|K_S|$ in each of the three cases $Ia$, $Ib$,
$II$. It turns out that if $S$ is either of type $Ia$ or of type
$II$ then the general curve in $|K_S|$ is irreducible, whereas if
$S$ is of type $Ib$ then $|K_S|=Z + |\Phi|$, where $|\Phi|$ is a
base-point free pencil of curves of genus $3$.

Finally, Section 3 is devoted to the proof of Theorem B. Such a
proof involves the calculation of the monodromy action of the
paramodular group $G_{\Delta}$ on the set $\widehat{A}[2]$, with
$\Delta=\left(
          \begin{array}{cc}
            1 & 0 \\
            0 & 2 \\
          \end{array}
        \right).$
This is probably well-known to the experts but, at least to our
knowledge, it is nowhere explicitly written, so we dedicated to it
an Appendix.

\bigskip
\textbf{Acknowledgments} M. Penegini was partially supported by
the DFG Forschergruppe 790 \emph{Classification of algebraic
surfaces and compact complex manifolds}.

F. Polizzi was partially supported by  Progetto MIUR di Rilevante
Interesse Nazionale \emph{Geometria delle Variet$\grave{a}$
Algebriche e loro Spazi di Moduli}.

Both authors are grateful to F. Catanese for stimulating discussions
related to this work and to C. Ciliberto for pointing out the
existence of surfaces of type $II$ and for useful remarks.

Moreover, it is a pleasure to thank J. S. Ellenberg and L.
Moret-Bailly for suggesting the computation in the Appendix (see the
\verb|mathoverflow| thread
http://mathoverflow.net/questions/49970/2-torsion-line-bundles-on-abelian-varieties).

Both authors wish to thank the referee for many detailed comments and suggestions
that considerably improved the presentation of these results.

\bigskip

\textbf{Notation and conventions.} We work over the field
$\mathbb{C}$ of complex numbers.

If $A$ is an abelian variety and $\wA:= \textrm{Pic}^0(A)$ is its
dual, we denote by $o$ and $\hat{o}$ the zero point of $A$ and
$\widehat{A}$, respectively. Moreover, $A[2]$ and $\widehat{A}[2]$
stand for the subgroups of $2$-division points.

If $\mathcal{L}$ is a line bundle on $A$ we denote by
$\phi_{\mathcal{L}}$ the morphism $\phi_{\mL} \colon A \rightarrow
\wA$ given by $x \mapsto t^*_x \mL \otimes \mL^{-1}$. If
$c_1(\mL)$ is non-degenerate then $\phi_{\mL}$ is an isogeny, and
we denote by $K(\mL)$ its kernel.

A coherent sheaf $\mathcal{F}$ on $A$ is called a \emph{IT-sheaf of
index i} if
\begin{equation*}
H^j(A, \, \mathcal{F} \otimes \mathcal{Q})=0 \quad \textrm{for all } \mathcal{Q} \in \textrm{Pic}^0(A)
\quad \textrm{and } j\neq i.
\end{equation*}
If $\mathcal{F}$ is an IT-sheaf of
index $i$ and $\mathcal{P}$ it the normalized Poincar\'e bundle on $A \times \wA$, the coherent sheaf
\begin{equation*}
\widehat{\mathcal{F}}:={R}^i\pi_{\wA \, *}(\mathcal{P} \otimes
\pi^*_{A}\mathcal{F})
\end{equation*}
is a vector bundle of rank $h^i(A, \, \mathcal{F})$, called the \emph{Fourier-Mukai transform} of $\mathcal{F}$.

By ``\emph{surface}'' we mean a projective, non-singular surface $S$, and
for such a surface $\omega_S=\oo_S(K_S)$ denotes the canonical
class, $p_g(S)=h^0(S, \, \omega_S)$ is the \emph{geometric genus},
$q(S)=h^1(S, \, \omega_S)$ is the \emph{irregularity} and
$\chi(\mathcal{O}_S)=1-q(S)+p_g(S)$ is the \emph{Euler-Poincar\'e
characteristic}. If $q(S)>0$, we denote by $\alpha \colon S \to
\textrm{Alb}(S)$ the Albanese map of $S$.

Throughout the paper, we denote Cartier divisors on a variety by
capital letters and the corresponding line bundles by italic
letters, so we write for instance $\mathcal{L}=\oo_S(L)$.

If $|L|$ is any complete linear system of curves on a surface, its
base locus is denoted by $\textrm{Bs}|L|$.

If $X$ is any scheme, by ``first-order deformation" of $X$ we mean a
deformation over $\textrm{Spec}\,
\mathbb{C}[\epsilon]/(\epsilon^2)$, whereas by ``small deformation"
we mean a deformation over a disk $\mathscr{D}_r=\{ t \in \mathbb{C}
\, | \, |t| < r \}$.

In Section \ref{sec.prel} we use the following special case of
Eagon-Northcott complex. Let us consider a short exact sequence
 of sheaves on $S$ of the form
\begin{equation*}
0 \lr \mathcal{L} \lr \FF \lr \mathcal{M} \otimes \mathcal{I}_p
\to 0,
\end{equation*}
 where $\mL$, $\mathcal{M}$ are line bundles, $\FF$ is a rank $2$
vector bundle and $p$ is a point. Then the symmetric powers $S^2
\FF$ and $S^3 \FF$ fit into short exact sequences
\begin{equation*}
\begin{split}
 0 & \lr \FF \otimes \mL \lr S^2 \FF \lr \mathcal{M}^2
 \otimes \mathcal{I}^2_p\lr 0, \\
 0 & \lr S^2 \FF \otimes \mL \lr S^3 \FF \lr \mathcal{M}^3
 \otimes \mathcal{I}^3_p\lr 0.
\end{split}
\end{equation*}

\section{Abelian surfaces with $(1, \,2)$-polarization} \label{sec.prel}


In this section we prove some technical facts about abelian surfaces
with  polarization of type $(1, \, 2)$ which are needed in the
sequel of the paper. The crucial results are Proposition
\ref{prop.2L-1}, Proposition \ref{prop.2L-2}, Corollary
\ref{cor.sing.curv} and Corollary \ref{cor.restriction}. For the
statements whose proof is omitted we refer the reader to
\cite{Ba87}, \cite{HvM89}, \cite[Chapter 10]{BL04}, \cite{BPS09} and
\cite{PP10}.

Let $A$ be an abelian surface and $L$ an ample divisor on $A$ with
$L^2=4$. Then $L$ defines a positive definite line bundle $\mL:=\oo_A(L)$ on $A$, whose first Chern class is a polarization of type
$(1, \, 2)$. By abuse of notation we consider the line bundle $\mL$ itself as a polarization. Moreover we have $h^0(A, \, \mL)=2$ so the linear system
$|L|$ is a pencil.

\begin{proposition} \emph{\cite[p. 46]{Ba87}} \label{prop.barth.class}
Let $(A, \, \mL)$ be a $(1, \, 2)$-polarized abelian surface and
let $G \in |L|$. Then we are in one of the following cases:
\begin{itemize}
\item[$\boldsymbol{(a)}$] $G$ is a smooth, connected curve of
genus $3;$ \item[$\boldsymbol{(b)}$] $G$ is an irreducible curve of
geometric genus $2$, with an ordinary double point$;$
\item[$\boldsymbol{(c)}$] $G=E+F$, where $E$ and $F$ are elliptic
curves and $EF=2;$ \item[$\boldsymbol{(d)}$] $G=E+F_1+F_2$, with
$E$, $F_1$, $F_2$ elliptic curves such that $EF_1=1$, $EF_2=1$,
$F_1F_2=0$.
\end{itemize}
Moreover, in case $(c)$ the surface $A$ admits an isogeny onto a
product of two elliptic curves, and the polarization of $A$ is the
pull-back of the principal product polarization, whereas in case
$(d)$ the surface $A$ itself is a product $E \times F$ and
$\mL=\mathcal{O}_A(E+2F)$.
\end{proposition}

Let us denote by $\Delta$ the matrix $\left(
                                        \begin{array}{cc}
                                          1 & 0 \\
                                          0 & 2 \\
                                        \end{array}
                                      \right)$,
and by $\mathcal{A}_{\Delta}$ the moduli space of
$(1,\,2)$-polarized abelian surfaces; then there exists a Zariski
dense set $\mathcal{U} \subset \mathcal{A}_{\Delta}$ such that,
given any $(A, \, \mL) \in \mathcal{U}$, all divisors in $|L|$ are
irreducible, i.e., of type $(a)$ or $(b)$, see \cite[Section
3]{BPS09}.

\begin{definition}
If $(A, \, \mL) \in \mathcal{U}$, we say that $\mL$ is a
\emph{general} $(1, \,2)$-polarization. If $|L|$ contains some
divisor of type $(c)$, we say that $\mL$ is a \emph{special} $(1,
\,2)$-polarization. Finally, if the divisors in $|L|$ are of type
$(d)$, we say that $\mL$ is a \emph{product} $(1,
\,2)$-polarization.
\end{definition}

In the rest of this section we assume that $\mL$ is not a product
polarization. Then $|L|$ has four distinct base points $\{e_0,\,
e_1, \, e_2, \, e_3\}$, which form an orbit for the action of
$K(\mL)\cong (\ZZ/2\ZZ)^2$ on $A$. Moreover all curves in $|L|$
are smooth at each of these base points, see \cite[Section
1]{Ba87}.

Let us denote by $(-1)_A$ the involution $x \to -x$ on $A$. Then we
say that a divisor $C$ on $A$ is \emph{symmetric} if $(-1)^*_A C =
C$. Analogously, we say that a vector bundle $\mathcal{F}$ on $A$ is
symmetric if $(-1)^*_A \FF = \FF$.

Since $\mathcal{L}$ is ample, \cite{Ba87} implies that,
up to translations, we may suppose that $\mathcal{L}$ is symmetric
and that the base locus of $|L|$ coincides with $K(\mL)$. Moreover:
\begin{itemize}
\item for all sections $s \in H^0(A, \,
\mL)$ we have $(-1)^*_{A}s=s$. In particular, all divisors in $|L|$
are symmetric;
 \item we may assume that $e_0=o$
and that  $e_1, \, e_2, \, e_3$ are $2$-division points,
 satisfying $e_1+e_2=e_3$.
\end{itemize}
There exist exactly three $2$-torsion line bundles $\mathcal{Q}_1$,
$\mathcal{Q}_2$, $\mathcal{Q}_3$ on $A$, with $\mathcal{Q}_1 \otimes
\mQ_2 = \mQ_3$, such that the linear system $|L+Q_i|$ contains an
irreducible curve which is singular at $o$. More precisely, one
shows that $h^0(A, \, \mathcal{L} \otimes \mathcal{Q}_i \otimes
\mathcal{I}_o^2)=1$ and that the unique curve $N_i \in |L+Q_i|$
which is singular at $o$ actually has an ordinary double point
there.

Denoting by $\phi_2 \colon A[2] \to \widehat{A}[2]$ the
 homomorphism induced by $\phi_{\mathcal{L}} \colon A \to \widehat{A}$ on the subgroups
 of $2$-division
 points, both $\ker \phi_2$ and $\textrm{im} \, \phi_2$ are isomorphic
 to $(\ZZ/2\ZZ)^2$. Indeed, we have
\begin{equation*}
\ker \phi_2 = K(\mL) \quad \textrm{and} \quad  \textrm{im}\,
 \phi_2=\{\oo_A, \, \mathcal{Q}_1, \, \mathcal{Q}_2, \,
 \mathcal{Q}_3 \}.
\end{equation*}
Let $\textrm{im}\, \phi_2^{\,\times}$ be the set $\{\mathcal{Q}_1,
\, \mathcal{Q}_2, \, \mathcal{Q}_3\}$.

\begin{proposition} \label{prop.Q.nodal}
Let $\mathcal{Q}, \, \mathcal{Q}' \in \widehat{A}$ and $p \in
\emph{Bs}\,|L+Q|$. Then
\begin{equation*}
h^0(A, \,\mL \otimes \mQ' \otimes \mathcal{I}_p^2)=\left\{
  \begin{array}{ll}
    0 & \emph{if }\mQ' \otimes \mQ^{-1} \notin \emph{im}\, \phi_2^{\,\times} \\
    1 & \emph{if }\mQ' \otimes \mQ^{-1} \in \emph{im}\, \phi_2^{\,\times}.
  \end{array}
\right.
\end{equation*}
\end{proposition}
\begin{proof}
Since $p \in \textrm{Bs}\,|L+Q|$, translating by $p$ we see that
$h^0(A, \,\mL \otimes \mQ' \otimes \mathcal{I}_p^2) \neq 0$ if and
only if $h^0(A, \,\mL \otimes \mQ' \otimes \mQ^{-1} \otimes
\mathcal{I}_o^2) \neq 0.$ Now the claim follows because this holds
precisely when $\mQ' \otimes \mQ^{-1} \in \textrm{im}\,
\phi_2^{\,\times}$.
\end{proof}

For any $\mathcal{Q} \in \widehat{A}$, let us consider the linear
system $|\mL^2 \otimes \mathcal{Q} \otimes
\mathcal{I}_o^4|:=\mathbb{P}H^0(A, \, \mL^2 \otimes \mathcal{Q}
\otimes \mathcal{I}_o^4)$ consisting of the curves in $|2L+Q|$
having a point of multiplicity at least $4$ at $o$. We first analyze
the case $\mQ=\oo_A$.

\begin{proposition} \label{prop.2L-1}
We have $h^0(A, \, \mathcal{L}^2 \otimes \mathcal{I}_o^4)=2$, that
is the linear system $|\mL^2 \otimes \mathcal{I}_o^4| \subset |2L|$
is a pencil. Moreover, if $C \in |\mL^2 \otimes \mathcal{I}_o^4|$
then we are in one of the following cases:
\begin{itemize}
\item[$\boldsymbol{(a)}$] $C$ is an irreducible curve of geometric
genus $3$, with an ordinary quadruple point $($this corresponds to
the general case$);$ \item[$\boldsymbol{(b)}$] $C$ is an
irreducible curve of geometric genus $2$, with an ordinary
quadruple point and an ordinary double point$;$
\item[$\boldsymbol{(c)}$] $C=2C'$, where $C'$ is an irreducible
curve of geometric genus $2$ with an ordinary double point$;$
\item[$\boldsymbol{(d)}$] $\mathcal{L}$ is a special
$(1,\,2)$-polarization and $C=2C'$, where $C'$ is the union of two
elliptic curves intersecting transversally in two points.
\end{itemize}
\end{proposition}
\begin{proof}
Since the three curves $2N_i$ belong to $|\mathcal{L}^2 \otimes
\mathcal{I}^4_o|$ and each $N_i$ is irreducible, by Bertini
theorem it follows that the general element $C \in |\mathcal{L}^2
\otimes \mathcal{I}^4_o|$ is irreducible and smooth outside $o$.
On the other hand we have $(2L)^2=16$, so $C$ has an
\emph{ordinary} quadruple point at $o$. Blowing up this point, the
strict transform of $C$ has self-intersection $0$, so
$|\mathcal{L}^2 \otimes \mathcal{I}^4_o|$ is a pencil.

Assume first that $\mL$ is a general polarization. We have shown
that the general curve in $|\mathcal{L}^2 \otimes
\mathcal{I}^4_o|$ belongs to case $(a)$. In order to complete the
proof, observe that $|\mathcal{L}^2 \otimes \mathcal{I}^4_o|$
contains the following distinguished elements:
\begin{itemize}
\item three reduced, irreducible curves $B_1$, $B_2$, $B_3$ such
that $B_i$ has an ordinary quadruple point at $o$, an ordinary
double point at $e_i$ and no other singularities (see
\cite[Corollary 4.7.6]{BL04}). These curves are as in case $(b)$;
\item three non-reduced elements, namely $2N_1$, $2N_2$, $2N_3$. These
 curves are as in case $(c)$.
\end{itemize}

Moreover, all other elements of $|\mathcal{L}^2 \otimes
\mathcal{I}^4_o|$ are smooth outside $o$; this can be seen by
blowing-up $o$ and applying the Zeuthen-Segre formula to the
fibration induced by the strict transform of the pencil, see
\cite[Section 1.2]{PP10}.

Finally, assume that $\mL$ is a special polarization. Then there is
just one more possibility, namely $C=2C'$, where $C'$ is the
translate of a reducible curve $E+F \in |L|$ by a suitable
$2$-division point. This yields case $(d)$.
\end{proof}

Let us consider now the case where $\mQ$ is non-trivial. In the
sequel, $\{i, \,j, \,k\}$ always denotes a permutation of $\{1, \,
2, \,3\}$.

\begin{proposition} \label{prop.2L-2}
Let $\mathcal{Q} \in \widehat{A}$ be non-trivial. Then
$|\mathcal{L}^2 \otimes \mQ \otimes \mathcal{I}^4_o|$ is empty,
unless $\mathcal{Q} \in
 \emph{im}\;\phi_2^{\,\times}$. More precisely, for all $i \in \{1,\, 2, \,3\}$
 we have  $h^0(A, \, \mL^2
 \otimes \mQ_i \otimes \mathcal{I}_o^4)=1$, so that
$|\mathcal{L}^2 \otimes \mQ_i \otimes \mathcal{I}^4_o|$ consists
of a unique element, namely the curve $N_j+N_k$.
\end{proposition}
\begin{proof}
Assume that there exists an effective curve $C \in |\mathcal{L}^2
\otimes \mQ \otimes \mathcal{I}^4_o|$. Blowing up the point $o$, the
strict transform $\widetilde{C}$ of $C$ is numerically equivalent to
the strict transform of a general element of the pencil
$|\mathcal{L}^2 \otimes \mathcal{I}^4_o|$. Since $\mathcal{Q}$ is
non-trivial, by the description of the non-reduced elements of $|\mL^2
\otimes \mathcal{I}_o^4|$ given in Proposition \ref{prop.2L-1} we
must have $\widetilde{C}=\widetilde{N}_j+\widetilde{N}_k$, so
$\mathcal{Q}=\mathcal{Q}_i$.
\end{proof}

Summing up, Propositions \ref{prop.2L-1} and \ref{prop.2L-2} imply

\begin{corollary} \label{cor.2L-3}
Let $\mathcal{Q} \in \widehat{A}$. Then $|\mathcal{L}^2 \otimes
\mQ \otimes \mathcal{I}^4_o|$ is empty, unless $\mathcal{Q} \in
\emph{im}\;\phi_2$. In this case
\begin{equation*}
\dim |\mathcal{L}^2 \otimes \mQ \otimes \mathcal{I}^4_o|= \left\{
\begin{array}{ll}
    1 & \hbox{\emph{if} } \mathcal{Q}=\oo_A \\
    0 & \hbox{\emph{if} } \mathcal{Q} \in \emph{im}\, \phi_2^{\,\times}. \\
\end{array}
\right.
\end{equation*}
\end{corollary}

\begin{corollary} \label{cor.sing.curv}
Let $(A, \, \mL)$ be a $(1,\,2)$-polarized abelian surface, and
let $C$ be a \emph{reduced} divisor numerically equivalent to $2L$
which has a quadruple point at $p \in A$. Then $C$ belongs to one
of the following cases, all of which occur:
\begin{itemize}
\item[$\boldsymbol{(i)}$] $C$ is irreducible, with an ordinary
quadruple point at $p$ and no other singularities$;$
\item[$\boldsymbol{(ii)}$] $C$ is irreducible, with an ordinary
quadruple point at $p$, an ordinary double point and no other
singularities$;$ \item[$\boldsymbol{(iii)}$] $C=C_1+ C_2$, where
$C_i$ is irreducible and numerically equivalent to $L$, with an
ordinary double point at $p$ and no other singularities. Since
$C_1C_2=4$, the singularity of $C$ at $p$ is again an ordinary
quadruple point.
\end{itemize}
\end{corollary}

\begin{proposition}
There exists a rank $2$, indecomposable vector bundle  $\FF$ on $A$,
such that
\begin{equation} \label{eq.coh.F}
h^0(A, \, \FF)=1, \quad h^1(A, \, \FF)=0, \quad h^2(A, \, \FF)=0,
\end{equation}
\begin{equation*} \label{eq.chern.F}
c_1(\mathcal{F})=\mathcal{L}, \quad c_2(\mathcal{F})=1.
\end{equation*}
Moreover $\FF$ is symmetric and it is isomorphic to the unique
locally free extension of the form
\begin{equation} \label{eq.F}
0 \lr \oo_A \lr \FF \lr \mL \otimes \mathcal{I}_o \lr 0.
\end{equation}
\end{proposition}
\begin{proof}
Let $\mathcal{L}^*$ be the  $(1, \, 2)$-polarization on
$\widehat{A}$ which is dual to $\mL$. Then $\mathcal{L}^{* \, -1}$
is a IT-sheaf of index $2$ and its Fourier-Mukai transform
$\FF:=\widehat{\mathcal{L}^{* \, -1}}$ is a rank $2$ vector bundle
on $A$, which satisfies \eqref{eq.coh.F} by \cite[Theorem
14.2.2]{BL04} and \cite[Corollary 2.8]{Mu81}. In addition,
\cite[Proposition 14.4.3]{BL04} implies
$c_1(\mathcal{F})=\mathcal{L}$. Finally, Hirzebruch-Riemann-Roch
implies $c_2(\mathcal{F})=1$ and by \cite[Proposition 2.2]{PP10} and
\cite[Proposition 2.4]{PP10}, since $\mL$ is not a product
polarization, we infer that $\FF$ is symmetric and isomorphic to the
unique locally free extension \eqref{eq.F}.
\end{proof}

\begin{proposition} \label{prop.coh.S2.F}
Let $\mathcal{Q} \in \widehat{A}$. The following holds:
\begin{itemize}
\item[$\boldsymbol{(i)}$] if $\mathcal{Q} \notin \emph{im}\;
\phi_2^{\,\times}$, then
\begin{equation*}
h^0(A, \,S^2 \FF \otimes \bigwedge^2 \FF^{\vee} \otimes
\mathcal{Q})=0, \quad h^1(A, \,S^2 \FF \otimes \bigwedge^2
\FF^{\vee} \otimes \mathcal{Q})=0, \quad h^2(A, \,S^2 \FF \otimes
\bigwedge^2 \FF^{\vee} \otimes \mathcal{Q})=0;
\end{equation*}
\item[$\boldsymbol{(ii)}$] if $\mathcal{Q} \in \emph{im}\;
\phi_2^{\,\times}$, then
\begin{equation*}
h^0(A, \,S^2 \FF \otimes \bigwedge^2 \FF^{\vee} \otimes
\mathcal{Q})=1, \quad h^1(A, \,S^2 \FF \otimes \bigwedge^2
\FF^{\vee} \otimes \mathcal{Q})=2, \quad h^2(A, \,S^2 \FF \otimes
\bigwedge^2 \FF^{\vee} \otimes \mathcal{Q})=1.
\end{equation*}
\end{itemize}
\end{proposition}
\begin{proof}
Tensoring \eqref{eq.F} with $\mathcal{Q}$ we obtain $h^0(A, \, \FF
\otimes \mathcal{Q})=1$, that is $\FF \otimes \mQ$ has a non-trivial
section. By \cite[Proposition 5 p. 33]{F98} there exists an
effective divisor $C$ and a zero-dimensional subscheme $W \subset A$
such that $\FF \otimes \mQ$ fits into a short exact sequence
\begin{equation} \label{eq.F.Q-2}
0 \lr \mathcal{C} \lr \FF \otimes \mQ \lr \mL \otimes \mQ^2
 \otimes \mathcal{C}^{-1} \otimes \mathcal{I}_W \lr 0,
\end{equation}
where $\mathcal{C}=\oo_A(C)$. Then $h^0(A, \, \mathcal{C})=1$ and
\begin{equation} \label{eq.c2}
1=c_2(\FF \otimes \mQ)=C(L-C) + \ell(W).
\end{equation}
Now there are three possibilities:
\begin{itemize}
\item[$(i)$] $C$ is an elliptic curve; \item[$(ii)$] $C$ is a
principal polarization; \item[$(iii)$] $C=0$.
\end{itemize}
We want to show that $(i)$ and $(ii)$ cannot occur.

In case $(i)$ we have $C^2=0$, then by \eqref{eq.c2} we obtain
$CL=1$ and $\ell(W)=0$. Thus \cite[Lemma 10.4.6]{BL04} implies that
$\mL$ is a product polarization, contradiction.

In case $(ii)$, the Index Theorem yields $(CL)^2 \geq C^2 L^2 =8$,
so using \eqref{eq.c2} we deduce $CL=3$, $\ell(W)=0$. Tensoring
\eqref{eq.F.Q-2} by $\mQ^{-1}$ and setting
$\mathcal{C'}:=\mathcal{C} \otimes \mQ^{-1}$, we obtain
\begin{equation}
0 \lr \mathcal{C}' \lr \FF \lr \mL \otimes
  \mathcal{C}'^{-1}\lr 0.
\end{equation}
Since $\mathcal{C}'$ is also principal polarization, by applying the
same argument used in \cite[proof of Proposition 2.2]{PP10} we
conclude again that $\mathcal{L}$ must be a product polarization.

Therefore the only possibility is $(iii)$, namely $C=0$. It follows
$\ell(W)=1$, that is $W$ consists of a unique point $p \in A$ and
\eqref{eq.F.Q-2} becomes
\begin{equation} \label{eq.F.Q}
0 \lr \oo_A \lr \FF \otimes \mQ \lr \mL \otimes \mQ^2 \otimes
\mathcal{I}_p \lr 0.
\end{equation}
Moreover, since $\FF \otimes \mQ$ is locally free, we have $p \in
\textrm{Bs}\,|L+2Q|$, see \cite[Example 1.7]{Ca90} or
\cite[Theorem 12 p. 39]{F98} . Applying the Eagon-Northcott
complex to \eqref{eq.F.Q} and tensoring with $\bigwedge^2
\FF^{\vee} \otimes \mQ^{-1}$ we get
\begin{equation} \label{eq.S2F.Q}
0 \lr \FF^{\vee} \lr S^2 \FF \otimes \bigwedge^2 \FF^{\vee}
\otimes \mQ \lr \mL \otimes \mQ^3 \otimes \mathcal{I}_p^2 \lr 0,
\end{equation}
 hence
\begin{equation*}
 h^0(A, \,S^2 \FF \otimes \bigwedge^2 \FF^{\vee} \otimes
\mQ)=h^0(A, \,\mL \otimes \mQ^3 \otimes \mathcal{I}_p^2).
\end{equation*}
On the other hand, since $p \in \textrm{Bs}\,|L+2Q|$, Proposition
\ref{prop.Q.nodal} yields
\begin{equation} \label{eq.dimension}
h^0(A, \, \mL \otimes \mathcal{Q}^3 \otimes
\mathcal{I}_p^2)=\left\{
  \begin{array}{ll}
    0 & \hbox{if } \mQ \notin \textrm{im}\,\phi^{\times}_2 \\
    1 & \hbox{if } \mQ \in \textrm{im}\,\phi^{\times}_2.
  \end{array}
\right.
\end{equation}
Using Serre duality, the isomorphism $\FF^{\vee} \cong \FF \otimes
\bigwedge^2 \FF^{\vee}$ and \eqref{eq.dimension}, since $\mQ \in
\textrm{im}\, \phi_2^{\times}$ if and only if $\mQ^{-1} \in
\textrm{im}\, \phi_2^{\times}$ we obtain
\begin{equation*}
h^2(A, \, S^2\FF \otimes \bigwedge^2 \FF^{\vee} \otimes
\mQ)=h^0(A, \, S^2\FF \otimes \bigwedge^2 \FF^{\vee} \otimes
\mQ^{-1})=h^0(A, \, S^2\FF \otimes \bigwedge^2 \FF^{\vee} \otimes
\mQ)
\end{equation*}
for all $\mathcal{Q} \in \widehat{A}$. Moreover
Hirzebruch-Riemann-Roch gives $\chi(A, \, S^2\FF \otimes \bigwedge^2
\FF^{\vee} \otimes \mathcal{Q})=0$, hence we get
\begin{equation*}
 h^1(A, \, S^2\FF \otimes \bigwedge^2 \FF^{\vee} \otimes
\mQ)=2 \cdot h^0(A, \, S^2\FF \otimes \bigwedge^2 \FF^{\vee}
\otimes \mQ).
\end{equation*}
This completes the proof.
\end{proof}

\begin{proposition} \label{prop.coh.S3.F}
For any $\mathcal{Q} \in \widehat{A}$, we have
\begin{equation*}
h^0(A, \,S^3 \FF \otimes \bigwedge^2 \FF^{\vee} \otimes
\mathcal{Q})=2, \quad h^1(A, \,S^3 \FF \otimes \bigwedge^2
\FF^{\vee} \otimes \mathcal{Q})=0, \quad h^2(A, \,S^3 \FF \otimes
\bigwedge^2 \FF^{\vee} \otimes \mathcal{Q})=0.
\end{equation*}
\end{proposition}
\begin{proof}
By Hirzebruch-Riemann-Roch we obtain $\chi(A, \, S^3 \FF \otimes
\bigwedge^2 \FF^{\vee} \otimes \mathcal{Q})=2$, so it suffices to
show that $h^i(A, \, S^3 \FF \otimes \bigwedge^2 \FF^{\vee}
\otimes \mathcal{Q})=0$ for $i=1, \, 2$. The sheaf $\FF \otimes
\mQ^{-1}$ satisfies IT of index $0$ and $h^0(A, \, \FF \otimes
\mQ^{-1})=1$, so its Fourier-Mukai transform
$\mathcal{L}^{-1}_{\delta}:=\widehat{\FF \otimes \mQ^{-1}}$ is a line
bundle on $\wA$, which satisfies IT of index $2$ by \cite[Theorem 14.2.2]{BL04} and has $h^2(\wA, \, \mLd^{-1})=2$ by \cite[Corollary 2.8]{Mu81}. This means that
$\mL_{\delta}=(\mLd^{-1})^{-1}$ is a $(1, \, 2)$ polarization. Since $\mathcal{F}$
is a symmetric vector bundle, by using \cite[Corollary 2.4]{Mu81}
we obtain
\begin{equation*}
\widehat{\mLd}=(-1)_A^*(\FF \otimes \mQ^{-1})=\FF \otimes \mQ,
\end{equation*}
that is the rank $2$ vector bundle $\FF \otimes \mQ$ is the
Fourier-Mukai transform of $\mLd$. Therefore, taking the isogeny
\begin{equation*}
\phi = \phi_{\mLd^{-1}} \colon \wA \lr A
\end{equation*}
and using \cite[Proposition 3.11]{Mu81}, we can write
\begin{equation} \label{eq.mukai.Ld}
\phi^*(\FF \otimes \mQ)=\mLd \oplus \mLd.
\end{equation}
On the other hand, $\phi$ is a finite map so we have
\begin{equation*} H^i(A, \, S^3 \FF \otimes \bigwedge^2
\FF^{\vee} \otimes \mathcal{Q}) \cong \phi^*H^i(A, \, S^3 \FF
\otimes \bigwedge^2 \FF^{\vee} \otimes \mathcal{Q}) \subseteq
H^i(\wA, \, \phi^*(S^3 \FF \otimes \bigwedge^2 \FF^{\vee} \otimes
\mathcal{Q}) ).
\end{equation*}
Since
\begin{equation*}
S^3 \FF \otimes \bigwedge^2 \FF^{\vee} \otimes \mQ=S^3( \FF
\otimes \mathcal{Q}) \otimes \bigwedge^2(\FF \otimes \mQ)^{\vee},
\end{equation*}
by using \eqref{eq.mukai.Ld} we deduce
\begin{equation*}
H^i(\wA, \, \phi^*(S^3 \FF \otimes \bigwedge^2 \FF^{\vee} \otimes
\mathcal{Q}) )=H^i(\wA, \, \mLd)^{\oplus4}.
\end{equation*}
The right-hand side vanishes for $i=1, \,2$, so we are done.
\end{proof}
Let $\sigma \colon B \to A$ be the blow-up of $A$ at $o$ and let $E
\subset B$ be the exceptional divisor. Since $\textrm{Pic}^0(B)
\cong \sigma ^{*}\textrm{Pic}^0(A)$, by abusing notation we will
often identify degree $0$ line bundles on $B$ with degree $0$ line
bundles on $A$, and we will simply write $\mathcal{Q}$ instead of
$\sigma^* \mQ$.

The strict transform of the pencil $|\mL^2 \otimes
\mathcal{I}_o^4|$ gives the base-point free pencil
$|\sigma^*(2L)-4E|$ in $B$, whose general element is a smooth
curve of genus $3$.

\begin{proposition} \label{prop.restriction}
Let $D \in |\sigma^*(2L) - 4E|$ be a smooth curve and let
$\mathcal{Q} \in \widehat{A}$. Then $\oo_D(Q)=\oo_D$ if and only
if $\mQ \in \emph{im}\, \phi_2$.
\end{proposition}
\begin{proof}
If $\mQ=\oo_A$ the result is clear, so we assume that $\mQ \in
\widehat{A}$ is non-trivial. Since $h^1(B, \, \mQ)=h^2(B, \,
\mQ)=0$, by using the short exact sequence
\begin{equation*}
0 \lr \oo_B(Q - D) \lr \oo_B(Q) \lr \oo_D(Q) \lr 0
\end{equation*}
and Serre duality, we obtain
\begin{equation} \label{eq.Q.C}
\begin{split}
h^1(D, \, \oo_D(Q))&=h^2(B, \, \oo_B(Q-D))=h^0(B, \, \oo_B(D-Q+E))
\\
&=h^0(B, \, \sigma^* \oo_A(2L-Q) -3E)=h^0(A, \, \mL^2 \otimes
\mathcal{Q}^{-1} \otimes \mathcal{I}_o^3).
\end{split}
\end{equation}
In order to compute the last cohomology group, we will exploit the
vector bundle $\mathcal{F}$. In fact, applying the Eagon-Northcott
complex to \eqref{eq.F} and tensoring with
$\bigwedge^2\FF^{\vee}\otimes \mQ^{-1}$, we get
\begin{equation} \label{eq.S3F}
0 \lr S^2 \FF \otimes \bigwedge^2 \FF^{\vee} \otimes \mQ^{-1} \lr
S^3 \FF \otimes \bigwedge^2 \FF^{\vee} \otimes \mQ^{-1} \lr \mL^2
\otimes \mathcal{Q}^{-1} \otimes \mathcal{I}_o^3 \lr 0.
\end{equation}
By using \eqref{eq.S3F}, Proposition \ref{prop.coh.S2.F} and
Proposition \ref{prop.coh.S3.F}, we obtain
\begin{equation*}
h^0(A, \, \mL^2 \otimes \mathcal{Q}^{-1} \otimes
\mathcal{I}_o^3)=\left\{
  \begin{array}{ll}
    2 & \hbox{if } \mQ \notin \textrm{im}\,\phi^{\times}_2 \\
    3 & \hbox{if } \mQ \in \textrm{im}\,\phi^{\times}_2.
  \end{array}
\right.
\end{equation*}

Since $D$ is a smooth curve of genus $3$, by using \eqref{eq.Q.C}
and Riemann-Roch we deduce
\begin{equation*}
h^0(D, \, \oo_D(Q))=h^1(D, \, \oo_D(Q))-2 = \left\{
  \begin{array}{ll}
    0 & \hbox{if } \mQ \notin \textrm{im}\,\phi_2 \\
    1 & \hbox{if } \mQ \in \textrm{im}\,\phi_2.
  \end{array}
\right.
\end{equation*}
This completes the proof.
\end{proof}

\begin{corollary} \label{cor.restriction}
Let $\mQ \in \widehat{A}$ and let $D$ be a smooth curve in the
pencil $|\sigma^*(2L) - 4E|$. Then $\oo_D(\sigma^*(L+Q)-2E)=\oo_D$
if and only if $\mathcal{Q} \in \emph{im}\, \phi_2$.
\end{corollary}
\begin{proof}
For all $i \in \{1, \, 2, \, 3 \}$ the effective curve
$\widetilde{N}_i \in |\sigma^*(L+Q_i)-2E|$ does not intersect $D$,
so $\oo_D(\sigma^*(L+Q_i)-2E)=\oo_D$. Hence
 Proposition \ref{prop.restriction} yields
 $\oo_D(\sigma^*L-2E)=\oo_D$ and the claim follows.
\end{proof}

\section{Surfaces with $p_g=q=2$, $K^2=6$ and Albanese map of degree
$2$} \label{sec.d=2}

In the sequel, $S$ will be a smooth minimal surface of Albanese
general type with $p_g=q=2$ and $\alpha \colon S \to A$ will be its
Albanese map, which we suppose of degree $2$. Let $D_A \subset A$ be the branch locus of $\alpha$
and let
\begin{equation*} \label{dia.alpha}
\xymatrix{
S \ar[r] \ar[dr]_{\alpha} & X \ar[d]^{f} \\
 & A}
\end{equation*}
be the Stein factorization of $\alpha$. Then $f \colon X \to A$ is
a finite double cover and, since $S$ is smooth, it follows that
$X$ is normal, see \cite[Chapter I, Theorem 8.2]{BHPV03}. In
particular $X$ has at most isolated singularities, hence the curve
$D_A$ is reduced.

\begin{proposition} \label{prop.branch}
Assume that $K_S^2=6$ and that the Albanese map $\alpha \colon S
\to A$ is a generically finite double cover. Then there exists a
polarization $\mathcal{L}_A =\oo_A(L_A)$ of type $(1, \, 2)$ on
$A$ such that $D_A$ is a curve in $|2L_A|$ whose unique
non-negligible singularity is an ordinary quadruple point $p$.
\end{proposition}
\begin{proof}
$D_A$ is linearly equivalent to $2L_A$ for some divisor $L_A$ in
$A$. There is a ``canonical resolution" diagram
\begin{equation} \label{dia.resolution}
\begin{CD}
\bar{S}  @> >> X\\
@V{\beta}VV  @VV f V\\
B @> {\sigma}>> A,\\
\end{CD}
\end{equation}
where $\bar{S}$ is smooth and $\sigma \colon B \to A$ is composed of
a series of blow-ups, see \cite[Chapter III, Section 7]{BHPV03}. Let
$x_1,x_2, \ldots, x_r$ be the centers of these blow-ups, and let
$E_i$ be the inverse image of $x_i$ on $B$ (with right
multiplicities such that $E_i E_j=- \delta_{ij}, \; K_B=\sigma^*
K_A+\sum_{i=1}^r E_i$. Then the branch locus $D_B$ of $\beta \colon
\bar{S} \to B$ is smooth and can be written as
\begin{equation} \label{hat B}
D_B=\sigma^* D_A- \sum_{i=1}^r d_i E_i,
\end{equation}
where the $d_i$  are even positive integers, say $d_i=2m_i$. Let us
recall a couple of definitions:
\begin{itemize}
\item a \emph{negligible singularity} of $D_A$ is a point $x_j$ such that
$d_j=2$, and $d_i \leq 2$ for any point $x_i$ infinitely near to
$x_j;$
\item a $[2d+1, 2d+1]$- \emph{singularity} of $D_A$ is a pair $(x_i, \, x_j)$ such that
$x_i$ belongs to the first infinitesimal neighbourhood of $x_j$ and $d_i=2d+2, \,
d_j=2d$.
\end{itemize}
For example, a double point and an ordinary triple point are
negligible singularities, whereas a $[3,3]$-point is not. By using
the formulae in \cite[p. 237]{BHPV03} we obtain

\begin{equation} \label{eq.res.can}
2=2 \chi(\oo_{\bar{S}})=L_A^2- \sum_{i=1}^r m_i(m_i-1), \quad
K_{\bar{S}}^2=2L_A^2-2 \sum_{i=1}^r  (m_i-1)^2,
\end{equation}
which imply
\begin{equation*}
6=K_S^2 \geq K_{\bar{S}}^2 = 4 + 2 \sum_{i=1}^r (m_i-1).
\end{equation*}
If $m_i=1$ for all $i$, then all the $x_i$ are negligible
singularities and \eqref{eq.res.can} gives $K_S^2=4$, a
contradiction. Then we can assume $m_1=2$, $m_2= \ldots =m_r=1$.
Therefore \eqref{eq.res.can} yields $L_A^2=4$, that is
$\mL_A:=\oo_A(L_A)$ is a polarization of type $(1, \, 2)$ on $A$.
Now we have two possibilities:
\begin{itemize}
\item[$\boldsymbol{(i)}$] $x_1$ is not infinitely near to $x_2$; then
$D_A \in |2L_A|$ contains an ordinary quadruple point $p$ and
(possibly) some negligible singularities;
\item[$\boldsymbol{(ii)}$] $x_1$ is infinitely near to $x_2$; then $D_A \in
|2L_A|$ contains a point $p$ of type $[3,\,3]$ and (possibly) some
negligible singularities.
\end{itemize}
But in case $(ii)$ the surface $\bar{S}$ contains a $(-1)$-curve,
hence $K_S^2=7$, a contradiction. Therefore $D_A$ must be a curve of
type $(i)$. The existence of such a curve was proven in Corollary
\ref{cor.sing.curv}, so we are done.
\end{proof}

\begin{rmk} \label{rem.[3,3]}
The argument used in the proof of Proposition
\emph{\ref{prop.branch}} shows that, if we were able to find a curve
in $|2L_A|$ with a singular point of type $[3,\,3]$, then we could
construct a surface $S$ with $p_g=q=2$ and $K_S^2=7$. Unfortunately,
at present we do not know whether such a curve exists.
\end{rmk}

\begin{proposition} \label{prop.no product}
$\mL_A$ is not a product polarization.
\end{proposition}
\begin{proof}
Assume by contradiction that $\mL_A$ is a product polarization. Then
$A=E \times F$, with natural projection maps $\pi_E \colon A \to E$
and $\pi_F \colon A \to F$, and $L \equiv E+2F$. Let $F_p$ be the
fibre of $\pi_E$ passing through $p$. Since $D_A$ has a quadruple
point at $p$ and $D_A F_p=2$,  B$\acute{\textrm{e}}$zout  theorem
implies that $F_p$ is a component of $D_A$. Similarly, since
$D_A-F_p$ has a triple point at $p$ and
 $(D_A-F_p) F_p=2$, it follows that $F_p$ is a component of $D_A-F_p$.
Therefore $D_A$ contains the curve $F_p$ with multiplicity at
least $2$, which is impossible since $D_A$ must be reduced.
\end{proof}

Up to a translation we can now suppose $p=o$ and by using Corollary
\ref{cor.2L-3} we can write $\mL_A^2 = \mL^2 \otimes \mQ$, where
$\mathcal{Q} \in \textrm{im}\,\phi_2$ and $\mL$ is a symmetric
polarization, not of product type, such that $h^0(A, \, \mL^2
\otimes \mathcal{I}_o^4)=2$.

In the rest of this section we assume for simplicity that $D_A$
contains no negligible singularities besides the quadruple point
$o$; this is an open condition, equivalent to the ampleness of
$K_S$. Hence the map $\sigma \colon B \to A$ is just the blow-up at
$o$, we have $\bar{S}=S$ and \eqref{dia.resolution} induces the
following commutative diagram
\begin{equation*} \label{dia.beta}
\xymatrix{
S \ar[r]^{\beta} \ar[dr]_{\alpha} & B \ar[r]^{\varphi} \ar[d]^{\sigma} & \mathbb{P}^1 \\
 & A},
\end{equation*}
where $\beta \colon S \to B$ is a \emph{finite} double cover and
$\varphi \colon B \to \mathbb{P}^1$ is the morphism induced by the
base-point free pencil $|\sigma^*(2L)-4E|$. The double cover $\beta$
is branched along a \emph{smooth} divisor
\begin{equation*}
D_B \in |\sigma^*(2L+Q)-4E|,
\end{equation*}
hence it is defined by a square root of $\oo_B(D_B)$, namely by
$\mL_B:=\oo_B(\sigma^*(L+Q^{1/2})-2E)$, where $\mathcal{Q}^{1/2}$
is a square root of $\mQ$.

\begin{proposition} \label{prop.invariants}
$S$ is a minimal surface of general type with $p_g=q=2$ and
$K_S^2=6$, unless $\mQ=\mathcal{Q}^{1/2}=\oo_A$. In the last case
we have instead $p_g=q=3$ and $K_S^2=6$.
\end{proposition}
\begin{proof}
Standard formulae for double covers (\cite[p. 237]{BHPV03}) give
$\chi(\oo_S)=1$ and $K_S^2=6$. Moreover we have
$\beta_*\omega_S=\omega_B \oplus (\omega_B \otimes \mL_B)$, hence
we obtain
\begin{equation} \label{eq.LB-1}
p_g(S)=h^0(B, \, \oo_B(E)) + h^0(B, \, \oo_B(\sigma^*(L +
Q^{1/2})-E))=1+h^0(A, \, \mL \otimes \mQ^{1/2} \otimes
\mathcal{I}_o).
\end{equation}
If $\mathcal{Q}^{1/2}$ is not trivial then $h^0(A, \,\mL \otimes
\mQ^{1/2} \otimes \mathcal{I}_o)=1$, otherwise $h^0(A, \,\mL
\otimes \mathcal{I}_o)=2$.
\end{proof}

\begin{rmk} \label{rem.pg=3}
If $\mQ=\mQ^{1/2}=\oo_A$ then $S$ is the symmetric product of a
smooth curve of genus $3$, see \emph{\cite{HP02}} and
\emph{\cite{Pi02}}. Let us give an alternative construction of the
double cover $f \colon S \to A$ in this particular case. Take a
smooth curve $C$ of genus $3$, admitting a double cover $\varphi
\colon C \to E$ onto an elliptic curve $E$. Let $o$ be the
identity in the group law of $E$ and for all $x \in C$ let us
denote by $x'$ the conjugate point of $x$ with respect to the
involution $C \to C$ induced by $\varphi$. Then
$S:=\emph{Sym}^2(C)$ contains the elliptic curve $Z := \{x+x' \, |
\, x \in C \}$ which is isomorphic to $E$. Moreover, there is a
morphism
\begin{equation*}
\begin{split}
& \bar{\alpha} \colon S \to \emph{Pic}^0(C) \quad given \, \, by \\
& \bar{\alpha}(x+y)=\oo_C(x+y-\varphi^*(o)).
\end{split}
\end{equation*}
Now take any point $x+x' \in Z$ and let
$a:=\varphi(x)=\varphi(x')$. We have
\begin{equation*}
\bar{\alpha}(x+x')= \oo_C(x+x'-\varphi^*(o))=\varphi^* \oo_E(a-o)
\in \varphi^* \emph{Pic}^0(E),
\end{equation*}
that is the induced map
\begin{equation*}
\alpha \colon S \to A:=\emph{Pic}^0(C)/\varphi^*\emph{Pic}^0(E)
\end{equation*}
contracts $Z$ to a point. Moreover $\alpha$ has generic degree
$2$; in fact $\alpha(x+y)=\alpha(x'+y')$ for all $x$,  $y \in C$.
\end{rmk}

Since we are interested in the case $p_g(S)=q(S)=2$, in the sequel
we always assume $\mathcal{Q}^{1/2} \neq \oo_A$. Summing up, we
have proven the following result.

\begin{theorem} \label{th.A}
Given an  abelian surface $A$ with a symmetric polarization $\mL$ of
type $(1, \,2)$, not of product type,  for any $\mQ \in \emph{im}\,
\phi_2$ there exists a curve $D_A \in |\mL^2 \otimes \mQ|$ whose
unique non-negligible singularity is an ordinary quadruple point at
the origin $o \in A$. Let $\mQ^{1/2}$ be a square root of $\mQ$, and
if $\mQ=\oo_A$ assume moreover $\mQ^{1/2} \neq \oo_A$. Then the
minimal desingularization $S$ of the double cover of $A$ branched
over $D_A$ and defined by $\mL \otimes \mQ^{1/2}$ is a minimal
surface of general type with $p_g=q=2$, $K_S^2=6$ and Albanese map
of degree $2$.

Conversely, every minimal surface of general type with $p_g=q=2$,
$K_S^2=6$ and Albanese map of degree $2$ can be constructed in
this way.

\end{theorem}

In order to proceed with the study of our surfaces, let us
introduce the following

\begin{definition} \label{def.type}
Let $S$ be a minimal surface of general type with $p_g=q=2$,
$K_S^2=6$ and Albanese map of degree $2$.
\begin{itemize}
\item If $\mQ =\oo_A$ we say that $S$ is a \emph{surface of type}
$I$. Furthermore, if $\mathcal{Q}^{1/2} \notin \textrm{im} \,
\phi_2^{\times}$ we say that $S$ is of type $Ia$, whereas if
$\mathcal{Q}^{1/2} \in \textrm{im} \, \phi_2^{\,\times}$  we say
that $S$ is of type $Ib$. \item If $\mQ \in
\textrm{im}\,\phi_2^{\, \times}$ we say that $S$ is a
\emph{surface of type} $II$.
\end{itemize}
\end{definition}

\begin{rmk} \label{rem.types}
If $S$ is a surface of type $I$, then $D_A$ is as in Corollary
$\emph{\ref{cor.sing.curv}}$, case $(i)$ or $(ii)$. If $S$ is a
surface of type $II$, then $D_A$ is as in Corollary
\emph{\ref{cor.sing.curv}}, case $(iii)$. See Figures
$\eqref{fig.type.I}$ and $\eqref{fig.type.II}$.
\end{rmk}
\begin{figure}[H]
\begin{center}
\includegraphics*[totalheight=4 cm]{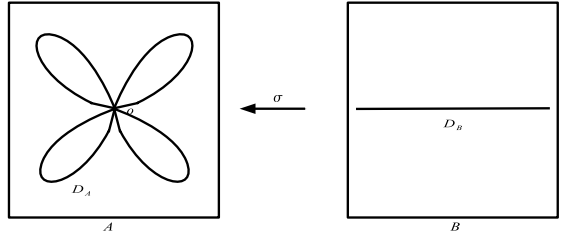}
\caption{The branch curves $D_A$ and $D_B$ for a general surface
of type $I$} \label{fig.type.I}
\end{center}
\end{figure}

\begin{figure}[H]
\begin{center}
\includegraphics*[totalheight=4 cm]{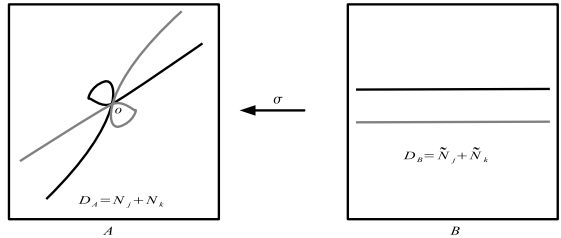}
\caption{The branch curves $D_A$ and $D_B$ for a surface of type
$II$} \label{fig.type.II}
\end{center}
\end{figure}

We denote by $R \subset S$ the ramification divisor of $\beta \colon
S \to B$ and
by $Z$ the divisor $\beta^*E$. Then $Z$ is an elliptic curve and $Z^2=-2$. \\

\begin{proposition} \label{prop.pencil.D}
The pullback via $\beta \colon S \to B$ of the general curve $D$ in
 the pencil $|D|=|\sigma^*(2L) - 4E|$ is reducible if and only if
 $S$ is of type $Ib$.
\end{proposition}
\begin{proof}
The restriction of $\beta$ to $D$ is the trivial double cover if and
only if $\mL_B \otimes \oo_D = \oo_D$, i.e. if and only if
$\oo_D(\sigma^*(L +Q^{1/2}) - 2E)= \oo_D$. Thus the result follows
from Corollary \ref{cor.restriction}.
\end{proof}

Now we want to describe the canonical system of our surfaces. Let
us analyze first surfaces of type $I$. Then $\mathcal{Q}^{1/2}$ is
a non-trivial, $2$-torsion line bundle and, for the general
surface $S$, the branch locus $D_B$ of $\beta \colon S \to B$ is a
smooth curve of genus $3$ belonging to the pencil
$|D|=|\sigma^*(2L) - 4E|$.

\begin{proposition} \label{prop.type I}
Let $S$ be a surface of type $I;$  then the following holds.
\begin{itemize}
\item[$\boldsymbol{(i)}$] If $S$ is of type $Ia$,
 the pullback via $\beta \colon S \to B$ of the pencil $|D|$ on $B$
is a base-point free pencil $|\Phi|$ on $S$, whose general element
$\Phi$ is a smooth curve of genus $5$ satisfying $\Phi Z=8$.
Moreover, the canonical system $|K_S|$ has no fixed part, hence
the general canonical curve of $S$ is irreducible. Finally, $2R
\in |\Phi|$. \item[$\boldsymbol{(ii)}$] If $S$ is of type $Ib$,
i.e. $\mathcal{Q}^{1/2}=\mQ_i$ for some $i \in \{1, \, 2,\, 3\}$,
there is a commutative diagram
\begin{equation} \label{dia.beta.mod}
\xymatrix{
S \ar[r]^{\beta} \ar[d]_{\phi} & B \ar[d]^{\varphi} \\
 \mathbb{P}^1 \ar[r]^b & \mathbb{P}^1,}
\end{equation}
where $b \colon \mathbb{P}^1 \to \mathbb{P}^1$ is a double cover
branched in two points, namely the point corresponding to the branch
locus $D_B$ and the point corresponding to the curve
$2\widetilde{N}_i$, where $N_i$ is the unique curves in $|\mL
\otimes \mathcal{Q}_i|$ with an ordinary double point at $o$ and
``$\thicksim$" stands for the strict transform in $B$. The general
fibre $\Phi$ of the map $\phi \colon S \to \mathbb{P}^1$ is a smooth
curve of genus $3$; moreover $Z$ is the fixed part of $|K_S|$ and
$|K_S|=Z + |\Phi|$, i.e. the canonical system is composed with the
pencil $|\Phi|$. Finally, $R \in |\Phi|$.
\end{itemize}
\end{proposition}
\begin{proof}
$(i)$ The fact that $\Phi$ is a smooth curve of genus $5$ follows
from Proposition \ref{prop.pencil.D}; moreover $\Phi
Z=(\beta^*D)(\beta^* E)=2DE=8$. We have $2R=\beta^*D_B \in |\Phi|$
and by Hurwitz formula $K_S=\beta^* K_B + R=Z+R$. Since $\dim
|K_S|=1$ and neither $Z$ nor $R$ move in a pencil, we deduce that
$|K_S|$ has no fixed part.

$(ii)$ If $S$ is of type $Ib$, then by Proposition
\ref{prop.pencil.D} the pull-back via $\beta$ of a general element
 of $|D|$ is the disjoint union of two smooth curves of genus $3$.
 So there exists a base-point free genus $3$ pencil $|\Phi|$ on $S$
and we obtain diagram \eqref{dia.beta.mod}.  In this case
$\mL_B=\oo_B(\widetilde{N}_i)$ is effective and it is no difficult
to see that $b \colon \mathbb{P}^1 \to \mathbb{P}^1$ is branched
only at the two points corresponding to $D_B$ and $N_i$. Moreover
$R=\beta^* \widetilde{N}_i \in |\Phi|$, so we can write
\begin{equation*}
|K_S|=|\beta^*K_B + R| = Z+|\Phi|,
\end{equation*}
that is the fixed part of the canonical pencil of $S$ is $|Z|$ and
its movable part is $|\Phi|$.
\end{proof}

Let us consider now surfaces of type $II$. Then $\mathcal{Q}=\mQ_i$
for some $i \in \{1, \, 2,\, 3\}$, so $\mQ^{1/2}$ is a degree $0$
line bundle whose order is exactly $4$ and the curve $D_B$ consists
of two distinct half-fibres of $|D|$, namely $D_B=\widetilde{N}_j +
\widetilde{N}_k$. Therefore $S$ is of type $II$ if and only
if $D_B$ is \emph{disconnected}. Proposition \ref{prop.pencil.D}
implies that the pullback via $\beta$ of a general curve in $|D|$ is
irreducible, so we obtain

\begin{proposition} \label{prop.type II}
If $S$ is a surface of type $II$, then the pullback via $\beta
\colon S \to B$ of the pencil $|D|$ on $B$ is a base-point free
pencil $|\Phi|$ on $S$, whose general element $\Phi$ is a smooth
curve of genus $5$ satisfying $\Phi Z=8$. Moreover, the canonical
system $|K_S|$ has no fixed part, hence the general canonical
curve of $S$ is irreducible. Finally, $R=R_1+R_2$ with $4R_1, \,
4R_2 \in |\Phi|$.
\end{proposition}
\begin{proof}
The first two parts of the statement follow from Proposition
\ref{prop.pencil.D} by the same argument used in the proof of
Proposition \ref{prop.type I}, part $(i)$. It remains only to
prove the assertion about $R$. Let $R_1$, $R_2$ be the two
effective curves in $S$ such that $\beta^* \widetilde{N}_j=2R_1$,
$\beta^* \widetilde{N}_k=2R_2$; then $R=R_1+R_2$. Moreover,  since
$\widetilde{N}_j$ and $\widetilde{N}_k$ are both half-fibres of
$|D|$, it follows $4R_1$, $4R_2 \in |\Phi|$ and we are done.
\end{proof}

\begin{rmk} \label{rem.K ample}
The general surface of type $I$ has ample canonical divisor. In addition, all surfaces of type $II$ have ample canonical divisor.
\end{rmk}

\section{The moduli space}

Let $S$ be a minimal surface of general type with $p_g=q=2$,
$K_S^2=6$ and Albanese map of degree $2$; for a general choice of
$S$ we may assume that $K_S$ is ample, see Remark \ref{rem.K ample}.
The following result can be found in \cite[Section 5]{Ca11}.
\begin{proposition} \label{prop.degree.alb} Let $S$ be a minimal surface of
general type with $q(S) \geq 2$
 and Albanese map $\alpha \colon S \to A$, and assume that $\alpha(S)$
is a surface. Then this is a topological property. If in addition
$q(S)=2$, then the degree of the $\alpha$ is a topological
invariant.
\end{proposition}
\begin{proof}
By \cite{Ca91} the Albanese map $\alpha$ induces a homomorphism of
cohomology algebras
\begin{equation*}\alpha^*\colon H^*(\textrm{Alb}(S), \, \mathbb{Z}) \longrightarrow H^*(S, \, \mathbb{Z})
\end{equation*}
and $H^*(\textrm{Alb}(S), \mathbb{Z})$ is isomorphic to the full
exterior algebra $\bigwedge^*  H^1(\textrm{Alb}(S), \,
\mathbb{Z}))\cong\bigwedge^*  H^1(S, \, \mathbb{Z})$. In particular,
if $q=2$ the degree of the Albanese map equals the index of the
image of $\bigwedge^4  H^1(S, \, \mathbb{Z})$ inside $H^4(S, \,
\mathbb{Z})$ and it is therefore a topological invariant.
\end{proof}

By Proposition \ref{prop.degree.alb} it follows that one may study
the deformations of $S$ by relating them to those of the flat double
cover $\beta \colon S \to B$. By \cite[p. 162]{Se06} we have an
exact sequence
\begin{equation} \label{suc.def.S}
0 \longrightarrow T_S \lr \beta^{*}T_B \lr \mathcal{N}_{\beta} \lr
0,
\end{equation}
where $\mathcal{N}_{\beta}$ is a coherent sheaf supported on $R$
called the \emph{normal sheaf of} $\beta$.

\begin{proposition} \label{prop.N-beta}
Assume that $K_S$ is ample. If $S$ is a surface of type $I$, then
$\mathcal{N}_{\beta} = \oo_R$. If $S$ is a surface of type $II$,
then $\mathcal{N}_{\beta}$ is a non-trivial, $2$-torsion element
of $\emph{Pic}^0(R)$.
\end{proposition}
\begin{proof}
Since $K_S$ is ample, $R$ is smooth and we have an isomorphism
\begin{equation*}
\mathcal{N}_{\beta} = (N_{R/S})^{\otimes 2}=\oo_R(2R),
\end{equation*}
see \cite[Lemma 3.2]{Rol10}. If $S$ is of type $I$ then either $R
\in |\Phi|$ or $2R \in |\Phi|$, see Proposition \ref{prop.type I},
so $\mathcal{N}_{\beta}$ is trivial. If $S$ is of type $II$ then $4R
\in |\Phi|$, see Proposition \ref{prop.type II}, so
$\mathcal{N}_{\beta}$ is a non-trivial, $2$-torsion line bundle.
\end{proof}

\begin{proposition} \label{prop.coh.TB}
Assume that $K_S$ is ample. Then the sheaf $\beta^*T_B$ satisfies
\begin{equation*}
h^0(S, \,\beta^*T_B)=0, \quad h^1(S, \,\beta^*T_B)=4, \quad h^2(S,
\,\beta^*T_B)=4.
\end{equation*}
\end{proposition}
\begin{proof}
Since $\beta \colon S \to B$ is a finite map, by using projection
formula and the Leray spectral sequence we deduce
\begin{equation} \label{eq.beta}
h^i(S, \, \beta^*T_B)=h^i(B, \, \beta_* \beta^* T_B)=h^i(B, \,
T_B)+h^i(B, \, T_B \otimes \mathcal{L}_B^{-1}), \quad i=0,\,1,\,2.
\end{equation}
There is a short exact sequence
\begin{equation} \label{eq.blow-up.tangent}
0 \lr T_B \to \sigma^*T_A \to \oo_E(-E) \to 0,
\end{equation}
see \cite[p. 73]{Se06}. Then a straightforward computation yields
\begin{equation} \label{eq.coh.TB}
h^0(B, \, T_B)=0, \quad h^1(B, \, T_B)=4, \quad h^2(B, \, T_B)=2.
\end{equation}
Now let us  tensor \eqref{eq.blow-up.tangent} with
$\mathcal{L}_B^{-1}$. Since $\sigma^*T_A=\oo_B \oplus \oo_B$ and
$\mathcal{L}_B^{-1} \otimes \oo_E(-E)=\oo_E(E)$, by taking
cohomology we obtain \
\begin{equation} \label{eq.TB(-L)}
h^i(B, \, T_B \otimes \mathcal{L}_B^{-1})= 2 \cdot h^i(B, \,
\mathcal{L}_B^{-1}), \quad i=0,\,1,\,2.
\end{equation}
Moreover we have\begin{equation} \label{eq.LB} h^0(B, \,
\mL_B^{-1})=0, \quad h^1(B,\, \mL_B^{-1})=0, \quad h^2(B, \,
\mL_B^{-1})=1,
\end{equation}
where the first equality comes from the fact that $D_A=2L_B$ is an
effective divisor, the third equality follows from Serre duality and
$h^0(B, \, \mL_B \otimes \oo_B(E))=1$, since $\mQ^{1/2} \neq \oo_A$,
and the second one is a consequence of Riemann-Roch.

Therefore the claim follows using \eqref{eq.beta},
\eqref{eq.coh.TB}, \eqref{eq.TB(-L)} and \eqref{eq.LB}.
\end{proof}

We have a commutative diagram
\begin{equation} \label{dia.TS}
\xymatrix{ & 0  \ar[d] & 0
 \ar[d] & 0  \ar[d]  & \\
0 \ar[r] & T_S \ar[r] \ar[d] & \beta^*T_B
\ar[r] \ar[d] & \mathcal{N}_{\beta} \ar[r] \ar[d]  & 0\\
0 \ar[r] & T_S \ar[r]\ar[d] & \alpha^*T_A \ar[r] \ar[d] &
\mathcal{N}_{\alpha}
\ar[r] \ar[d]& 0 \\
 & 0 \ar[r] & \oo_{Z}(-Z) \ar[r] \ar[d] & \oo_{Z}(-Z) \ar[r] \ar[d] & 0
 \\
&  & 0  & 0 &  }
\end{equation}
whose central column is the pullback of \eqref{eq.blow-up.tangent}
via $\beta \colon S \to B$.

\begin{proposition} \label{prop.h1-T} Let $S$ be a minimal
surface with $p_g=q=2$, $K_S^2=6$ and Albanese map of degree $2$,
and assume that $K_S$ is ample. Then
\begin{equation*}
h^1(S, \, T_S)=\left\{%
\begin{array}{ll}
    4 & \hbox{if $S$ is of type $I$} \\
    3 & \hbox{if $S$ is of type $II$.} \\
\end{array}%
\right.
\end{equation*}
\end{proposition}
\begin{proof}
Proposition \ref{prop.coh.TB} yields $H^0(S, \, \beta^*T_B)=0$,
 so looking at the central column of diagram \eqref{dia.TS} we obtain
 the long exact sequence in cohomology
\begin{equation*}
0  \longrightarrow H^0(S, \, \alpha^*T_A) \lr H^0(Z, \,
\oo_{Z}(-Z))  \lr H^1(S, \, \beta^*T_B)
\stackrel{\delta}{\longrightarrow} H^1(S, \, \alpha^*T_A) \lr 0.
\end{equation*}
Since $h^0(S, \, \alpha^*T_A)= h^0(Z, \, \oo_{Z}(-Z))=2$, it
follows that the map $\delta$ is an isomorphism. Therefore the
commutativity of \eqref{dia.TS} implies that the image of $H^1(S,
\, T_S)$ in $H^1(S, \, \beta^*T_B)$ coincides with the image of
$H^1(S, \, T_S)$ in $H^1(S, \, \alpha^*T_A) \cong H^1(A, \, T_A)$.
So we obtain the exact sequence
\begin{equation} \label{eq.TA}
0 \to H^0(R, \, \mathcal{N}_{\beta}) \lr H^1(S, \, T_S)
\stackrel{\gamma}{\lr} H^1(A, \, T_A).
\end{equation}
We claim that the image of $\gamma$ has dimension $3$. In order to
prove this, we borrow an argument from \cite[Section 6]{PP10}.
Take a positive integer $m \geq 2$ such that there exists a smooth
pluricanonical divisor $\Gamma \in |mK_S|$ and let $\Gamma'$ be
the image of $\Gamma$  in $A$. By \cite[Section 3.4.4 p.
177]{Se06}, the first order deformations of a pair $(X, \, Y)$,
where $X \subset Y$ is a closed subscheme and $Y$ is nonsingular,
are parameterized by the vector space $H^1(Y, \, T_Y \langle X
\rangle)$, where $T_Y \langle X \rangle$ is the sheaf of germs of
tangent vectors to $Y$ which are tangent to $X$. Notice that $T_Y
\langle X \rangle$ is usually denoted by $T_Y( - \log X)$ when $X$
is a normal crossing divisor with smooth components. In our
situation, a first-order deformation of the pair $(\Gamma, \, S)$
 induces a first-order deformation of the pair $(\Gamma', \, A)$,
because the the differential map $d \alpha \colon T_S \to T_A$ sends
vectors tangent to $\Gamma$ into vectors tangent to $\Gamma'$. Hence
we have a commutative diagram
\begin{equation*} \label{dia.proof.2}
\begin{xy}
\xymatrix{
H^1(S, \, T_S \langle  \Gamma \rangle)  \ar[d]_{\epsilon}
\ar[rr]^{\gamma'} & & H^1(A, \, T_A \langle \Gamma' \rangle)
 \ar[d]^{\epsilon'} \\
 H^1(S, \, T_S) \ar[rr]^{\gamma} & & H^1(A, \, T_A).  \\
  }
\end{xy}
\end{equation*}
Let us observe now the following facts.
\begin{itemize}
\item Since $S$ is smooth, the line bundle $\omega_S^m$ extends
along any first-order deformation of $S$, because the relative
dualizing sheaf is locally free for any smooth morphism of
schemes, see \cite[p. 182]{Man08}. Moreover, since $S$ is minimal
of general type, we have $h^1(S, \, \omega_S^m)=0$, so every
section of $\omega_S^m$ extends as well, see \cite[Section
3.3.4]{Se06}. This means that no first-order deformation of $S$
makes $\Gamma$ disappear, in other words $\epsilon$ is surjective.
Therefore $\textrm{im} \,\gamma \subseteq \textrm{im} \,
\epsilon'$. \item Since $(\Gamma')^2 >0$, the line bundle
$\oo_A(\Gamma')$ is ample on $A$; therefore it deforms along a
subspace of $H^1(A, \, T_A)$ of dimension $3$, see \cite[p.
152]{Se06}. Since every first-order deformation of the pair $(A,
\,\Gamma')$ induces a first-order deformation of the pair $(A, \,
\oo_A(\Gamma'))$, it follows that the image of $\epsilon'$ is at
most $3$-dimensional.
\end{itemize}
According to the above remarks, we obtain
\begin{equation*}
\dim \,(\textrm{im} \, \gamma) \leq \dim \,(\textrm{im} \,
\epsilon') \leq 3.
\end{equation*}
On the other hand, given any abelian surface $A$ with a $(1,
\,2)$-polarization, not of product type, by the results of Section
\ref{sec.d=2} we can construct a surface $S$ of type $I$ or $II$
such that $\textrm{Alb}(S)=A$. Then the dimension of $\textrm{im} \,
\gamma$ equals the dimension of the moduli space of $(1, \,
2)$-polarized abelian surfaces, which is precisely $3$. So
\eqref{eq.TA} implies
\begin{equation*}
h^1(S, \, T_S)=3 + h^0(R, \, \mathcal{N}_{\beta})
\end{equation*}
and by using Proposition \ref{prop.N-beta} we are done.
\end{proof}

By Proposition \ref{prop.degree.alb} we may consider the moduli
space $\mathcal{M}$ of minimal surfaces $S$ of general type with
$p_g=q=2$, $K_S^2=6$ and Albanese map of degree $2$. Let
$\mathcal{M}_{Ia}$, $\mathcal{M}_{Ib}$, $\mathcal{M}_{II}$ be the
subsets whose points parameterize isomorphism classes of surfaces of
type $Ia$, $Ib$, $II$, respectively. Therefore $\mathcal{M}$ can be
written as the disjoint union
\begin{equation*}
\mathcal{M}=\mathcal{M}_{Ia} \sqcup \mathcal{M}_{Ib} \sqcup
\mathcal{M}_{II}.
\end{equation*}
Set moreover $\mathcal{M}_I:=\mathcal{M}_{Ia} \sqcup
\mathcal{M}_{Ib}$.

\begin{proposition} \label{prop.moduli} The following holds:
\begin{itemize}
\item[$\boldsymbol{(i)}$] $\mathcal{M}_{Ia}$ and
$\mathcal{M}_{Ib}$ are irreducible, generically smooth of
dimension $4;$ \item[$\boldsymbol{(ii)}$] $\mathcal{M}_{II}$ is
irreducible, generically smooth of dimension $3$.
\end{itemize}
\end{proposition}
\begin{proof}
$\boldsymbol{(i)}$ The construction of a surface of type $I$ depends
on the following data:
\begin{itemize}
\item the choice of a $(1,\,2)$-polarized abelian surface $(A, \,
\mathcal{L})$, not of product type ;
\item the choice of a general divisor $D_A$ in the
 pencil $|\mL^2 \otimes \mathcal{I}_o^4|$;
\item the choice of a non-trivial line bundle $\mathcal{Q}$ such
that $\mathcal{Q}^2=\oo_A$.
\end{itemize}
Let $\mathcal{A}_{\Delta}[2]$ be the space of pairs $(A, \,
\mathcal{Q})$, where $A$ is a $(1,\,2)$-polarized abelian surface
and $\mathcal{Q} \in \widehat{A}$ is the isomorphism class of a
non-trivial, $2$-torsion line bundle. In the Appendix (see
Proposition \ref{prop.components}) we show that
$\mathcal{A}_{\Delta}[2]$ is a quasi-projective variety, union of
two connected,
 irreducible components of dimension $3$
\begin{equation*}
\mathcal{A}_{\Delta}^{(a)}[2] \quad  \textrm{and} \quad
\mathcal{A}_{\Delta}^{(b)}[2],
\end{equation*}
which correspond to $\mathcal{Q} \notin \textrm{im}\,
\phi_2^{\times}$ and $\mathcal{Q} \in \textrm{im}\,
\phi_2^{\times}$, respectively. Therefore there are two
generically finite, dominant maps
\begin{equation*}
\mathcal{P}^{(a)} \longrightarrow \mathcal{M}_{Ia}, \quad
\mathcal{P}^{(b)} \longrightarrow \mathcal{M}_{Ib},
\end{equation*}
where $\mathcal{P}^{(a)}$ and $\mathcal{P}^{(b)}$ are suitable
projective bundles on $\mathcal{A}_{\Delta}^{(a)}[2]$ and
$\mathcal{A}_{\Delta}^{(b)}[2]$; it follows that $\mathcal{M}_{Ia}$
and $\mathcal{M}_{Ib}$ are irreducible of dimension $4$. On the
other hand, Proposition \ref{prop.h1-T} implies that for a general
$[S] \in \mathcal{M}_I$ we have
\begin{equation*}
\dim T_{[S]} \mathcal{M}_{I}=h^1(S, \, T_S)=4.
\end{equation*}
This shows that both $\mathcal{M}_{Ia}$ and $\mathcal{M}_{Ib}$ are
 generically smooth.

\medskip

$\boldsymbol{(ii)}$ The construction of a surface of type $II$
depends on the following data:
\begin{itemize}
\item the choice of a $(1,\,2)$-polarized abelian surface $(A, \,
\mathcal{L})$, not of product type ; \item the choice of
$\mathcal{Q} \in \textrm{im} \, \phi^{\times}_2$, which yields the
unique curve $D_A \in | \mL^2 \otimes \mathcal{Q}
\otimes\mathcal{I}_o^4|$; \item the choice of a square root
$\mathcal{Q}^{1/2}$ of $\mathcal{Q}$.
\end{itemize}
Let $\mathcal{A}_{\Delta}[2, \, 4]$ be space of triplets $(A, \,
\mQ, \, \mQ^{1/2})$, where $A$ is the isomorphism class of a $(1, \,
2)$-polarized abelian surface, $\mQ \in \textrm{im}\,
\phi_2^{\times}$ and $\mQ^{1/2}$ is a square root of $\mQ$. In the
Appendix (see Proposition \ref{prop.monodromy.2}) we show that
 $\mathcal{A}_{\Delta}[2, \, 4]$ is a $3$-dimensional, irreducible
 quasi-projective variety.
We have a generically finite, dominant map
\begin{equation*}
\mathcal{A}_{\Delta}[2, \, 4] \longrightarrow \mathcal{M}_{II},
\end{equation*}
so $\mathcal{M}_{II}$ is irreducible of dimension $3$. On the other
hand, Proposition \ref{prop.h1-T} implies that for a general $[S]
\in \mathcal{M}_{II}$ we have
\begin{equation*}
\dim T_{[S]} \mathcal{M}_{II}=h^1(S, \, T_S)=3,
\end{equation*}
hence $\mathcal{M}_{II}$ is generically smooth.
\end{proof}

\begin{proposition} \label{prop.disjoint}
$\mathcal{M}_{Ia}$, $\mathcal{M}_{Ib}$ and $\mathcal{M}_{II}$ are
connected components of $\mathcal{M}$.
\end{proposition}
\begin{proof}
We proved that $\mathcal{M}$ is the disjoint union of three
irreducible, constructible sets
\begin{equation*}
\mathcal{M}=\mathcal{M}_{Ia} \sqcup \mathcal{M}_{Ib} \sqcup
\mathcal{M}_{II},
\end{equation*}
so it is sufficient to show that $\mathcal{M}_{Ia}$,
 $\mathcal{M}_{Ib}$, $\mathcal{M}_{II}$ are all open in
 $\mathcal{M}$.
In other words, given a flat family $\mathscr{S} \to \mathscr{D}$
over a small disk $\mathscr{D}$, such that $S_0 \in
\mathcal{M}_{Ia}$ (resp. $S_0 \in$ $\mathcal{M}_{Ib}$,
$\mathcal{M}_{II}$), we must show that $S_t \in \mathcal{M}_{Ia}$
 (resp. $S_t \in$ $\mathcal{M}_{Ib}$, $\mathcal{M}_{II}$) for $t \neq 0$.
 We may associate to the
family $\mathscr{S} \to \mathscr{D}$ the family $\mathscr{X} \to
\mathscr{D}$, whose fibre over $t \in \mathscr{D}$ is the Stein
factorization $X_t$ of $S_t$, that is the contraction of the
elliptic curve $Z_t \subset S_t$. By the previous results it follows
that, up to a base change, the family $\mathscr{S} \to \mathscr{D}$
is the double cover of a family $\mathscr{B} \to \mathscr{D}$ of
blow-ups $B_t$ of $(1, \, 2)$-polarized abelian surfaces and the
family $\mathscr{X} \to \mathscr{D}$ is the double cover of the
family $\mathscr{A} \to \mathscr{D}$, where $A_t$ is the minimal
model of $B_t$. Globalizing the results of Section \ref{sec.d=2} we
see that the polarizations $\mathcal{L}_t$ on the abelian surfaces
$A_t$ glue together in order to give an ample line bundle
$\mathscr{L}$ on $\mathscr{A}$ and that there exists a divisor
$\mathscr{D}_{\mathscr{B}}$ on $\mathscr{B}$ whose restriction to
the fibre $B_t$ is the branch locus $D_{B_t}$ of $\beta_t \colon S_t
\to B_t$. Moreover, we find a commutative diagram
\begin{equation*} \label{dia.beta-glob}
\xymatrix{
\mathscr{S} \ar[r]^{\beta} \ar[dr]_{\alpha} & \mathscr{B}  \ar[d]^{\sigma} \\
 & \mathscr{A}}
\end{equation*}
and a line bundle $\mathscr{Q} \in \textrm{Pic}^0(\mathscr{A})$ of
order $2$ such that
\begin{equation*}
\mathscr{D}_{\mathscr{B}} \cong \sigma^*(2 \mathscr{L}+
\mathscr{Q})-4 \mathscr{E},
\end{equation*}
where $\sigma \colon \mathscr{B} \to \mathscr{A}$ is the relative
blow-down and $\mathscr{E}$ is the exceptional divisor of $\sigma$.
We denote by $\mathcal{Q}_t$ the restriction of $\mathscr{Q}$ to
$A_t$.

Now let us consider separately the three cases.
\begin{itemize}
\item $\mathcal{M}_{II}$ \emph{is open in} $\mathcal{M}$. \\
It is equivalent to prove that $\mathcal{M}_I$ is closed in
$\mathcal{M}$, namely that  $S_t \in \mathcal{M}_{I}$ for $t \neq 0$
implies $S_0 \in \mathcal{M}_{I}$.  The condition $S_t \in
\mathcal{M}_{I}$ for $t \neq 0$
 implies that $D_{B_t}$ is connected for any $t \neq 0$; it follows
 that $D_{B_0}$ is also connected, hence $S_0$ is again a surface
of type $I$.
\item $\mathcal{M}_{Ia}$ \emph{is open in} $\mathcal{M}$. \\
Assume that  $S_0 \in \mathcal{M}_{Ia}$. By Proposition
\ref{prop.type I}, this is equivalent to say that the branch locus
$D_{B_0}$ of $\beta_0 \colon S_0 \to B_0$
 is connected and that $|K_{S_0}|$ is base-point free. Clearly these are both
open conditions, so $\mathcal{M}_{Ia}$ is open in $\mathcal{M}$.
\item $\mathcal{M}_{Ib}$ \emph{is open in} $\mathcal{M}$. \\
Assume that  $S_0 \in \mathcal{M}_{Ib}$. Then we have $(A_0, \,
\mathcal{Q}_0) \in \mathcal{A}_{\Delta}^{(b)}[2]$. By Proposition
\ref{prop.components} in the Appendix it follows that
$\mathcal{A}_{\Delta}^{(b)}[2]$ is a connected component of
$\mathcal{A}_{\Delta}[2]$, in particular it is open therein. Hence
$(A_t, \, \mathcal{Q}_t) \in \mathcal{A}_{\Delta}^{(b)}[2]$ for $t
\neq 0$, proving that $\mathcal{M}_{Ib}$ is open in $\mathcal{M}$.
Notice that the same argument gives an alternative proof of the fact
that $\mathcal{M}_{Ia}$ is open in $\mathcal{M}$, since
$\mathcal{A}_{\Delta}^{(a)}[2]$ is the other connected component of
$\mathcal{A}_{\Delta}[2]$.
\end{itemize}
This completes the proof of Proposition \ref{prop.disjoint}.
\end{proof}

Summing up, Propositions \ref{prop.moduli}, Proposition
\ref{prop.disjoint} and Remark \ref{rem.K ample} imply the
following result.

\begin{theorem} \label{thm.B} Let $\mathcal{M}$ be the moduli space of
minimal surfaces $S$ of general type with $p_g=q=2$, $K_S^2=6$ and
Albanese map of degree $2$. Then the following holds:
\begin{itemize}
\item[$\boldsymbol{(i)}$] $\mathcal{M}$ is the disjoint union of
three connected components, namely
\begin{equation*}
\mathcal{M}=\mathcal{M}_{Ia} \sqcup \mathcal{M}_{Ib} \sqcup
\mathcal{M}_{II};
\end{equation*}
\item[$\boldsymbol{(ii)}$] these are also irreducible components
of the moduli space of minimal surfaces of general type$;$
\item[$\boldsymbol{(iii)}$] $\mathcal{M}_{Ia}, \,
\mathcal{M}_{Ib}, \, \mathcal{M}_{II}$ are generically smooth, of
dimension $4, \, 4, \, 3$, respectively$;$
\item[$\boldsymbol{(iv)}$] the general surface in
$\mathcal{M}_{Ia}$ and $\mathcal{M}_{Ib}$ has ample canonical
class$;$ all surfaces in $\mathcal{M}_{II}$ have ample canonical
class.
\end{itemize}
\end{theorem}




%
%
\section*{Appendix. The spaces $\mathcal{A}_{\Delta}[2]$ and $\mathcal{A}_{\Delta}[2, \, 4]$ and their connected components}

First let us recall some well-known facts about the moduli space of
polarized abelian surfaces, that can be found for instance in
\cite[Chapter 8]{BL04}.

Let us denote by $\Delta$ the matrix $\left(
                                                                  \begin{array}{cc}
                                                                    1 & 0 \\
                                                                    0 & 2 \\
                                                                  \end{array}
                                                                \right)$
and let
\begin{equation*}
\mathfrak{H}_2:=\{Z \in M_2(\mathbb{C}) \, | \, {}^tZ=Z, \,
\textrm{Im}\,Z >0 \}
\end{equation*}
be the Siegel upper half-space. We define a polarized abelian
surface of type $\Delta$ with symplectic basis to be a triplet
\begin{equation*}
(A, \, H, \, \{\lambda_1, \lambda_2, \mu_1, \mu_2 \})
\end{equation*}
with $A=\mathbb{C}^2/\Lambda$ an abelian surface, $H$ a polarization
of type $\Delta$ on $A$ and $\{\lambda_1, \lambda_2, \mu_1, \mu_2
\}$ a basis of the lattice $\Lambda$, symplectic with respect to
$H$. Then any $Z \in \mathfrak{H}_2$ determines a polarized abelian
surface of type $\Delta$ with symplectic basis $(A_Z, \, H_Z, \,
\{\lambda_1, \lambda_2, \mu_1, \mu_2 \})$ as follows: just set
\begin{equation*}
\lambda_Z:=(Z, \, D) \mathbb{Z}^{2g}, \quad H_Z=(\textrm{Im}\,
Z)^{-1}
\end{equation*}
and let $\{\lambda_1, \lambda_2, \mu_1, \mu_2 \}$ be the columns of
the matrix $(Z, \,D)$. Moreover, there exists a universal family,
that is a holomorphic family $\mathcal{X}_{\Delta} \to
\mathfrak{H}_2$ parameterizing these objects, see \cite[Section
8.7]{BL04}.

If $Z, \, Z' \in \mathfrak{H}_2$, the polarized abelian surfaces
$(A_Z, \, H_Z)$ and $(A_{Z'}, \, H_{Z'})$ are isomorphic if and only
if $Z'=M \cdot Z$, where
\begin{equation*} G_{\Delta}:=\left(
      \begin{array}{cccc}
        \mathbb{Z} & \mathbb{Z} & \mathbb{Z} & 2 \mathbb{Z} \\
        2 \mathbb{Z} & \mathbb{Z} & 2 \mathbb{Z} & 2 \mathbb{Z} \\
        \mathbb{Z} & \mathbb{Z} & \mathbb{Z} & 2 \mathbb{Z} \\
        \mathbb{Z} & \frac{1}{2} \mathbb{Z} & \mathbb{Z} & \mathbb{Z} \\
      \end{array}
    \right) \cap \textrm{Sp}_4(\mathbb{Q})
\end{equation*}
is the full paramodular group (see \cite[Chapter 8]{BL04},
\cite{Mu99}) and the action is defined as follows: for any $M=\left(
         \begin{array}{cc}
           \alpha & \beta \\
           \gamma & \delta \\
         \end{array}
       \right) \in G_{\Delta}$
and $Z \in \mathfrak{H}_2$, we set
\begin{equation} \label{eq.action on H2}
M \cdot Z :=(\alpha Z+\beta)(\gamma Z + \delta)^{-1}.
\end{equation}
Notice that the following special matrices lie in $G_{\Delta}$:
\begin{equation*}
M_{\boldsymbol{b}}:=\left(
\begin{array}{cccc}
1 & 0 & b_{11} & 2b_{12}\\
0 & 1 & 2b_{12} & 2b_{22} \\
0 & 0 & 1 & 0 \\
0 & 0 & 0 & 1 \\
\end{array}
\right),
\end{equation*}
\begin{equation*}
M_{\boldsymbol{d}}:=\left(
\begin{array}{cccc}
d_{22} & -d_{21} & 0 & 0\\
-2d_{12} & d_{11} & 0 & 0 \\
0 & 0 & d_{11} & 2d_{12} \\
0 & 0 & d_{21} & d_{22} \\
\end{array}
\right) \; \textrm{with} \; \left(
\begin{array}{cc}
d_{11} & 2d_{12} \\
d_{21} & d_{22} \\
\end{array}
\right) \in \textrm{SL}_2(\mathbb{Z}),
\end{equation*}
\begin{equation*}
M_{\boldsymbol{1,\,2}}:=\left(
\begin{array}{cccc}
0 & 0 & 1 & 0\\
0 & 0 & 0 & 2 \\
-1 & 0 & 0 & 0\\
0 & -\frac{1}{2} & 0 & 0 \\
\end{array}
\right).
\end{equation*}

The action \eqref{eq.action on H2} is properly discontinuos, so the
moduli space $\mathcal{A}_{\Delta}$ of $(1,\,2)$-polarized abelian
surfaces is a quasi-projective variety of dimension $3$, obtained as
the quotient $\mathfrak{H}_g/G_{\Delta}$. Then $G_{\Delta}$ is the
orbifold fundamental group of $\mathcal{A}_{\Delta}$ and there is an
induced monodromy action of $G_{\Delta}$ on both $A[2]$ and
$\widehat{A}[2]$, see \cite{Har79}.

\begin{proposition} \label{prop.monodromy}
The monodromy action of $G_{\Delta}$ on $\widehat{A}[2]$ has
precisely three orbits, namely
\begin{equation*}
\{\mathcal{O}_A\}, \quad \emph{im}\,\phi_2^{\times}, \quad
\widehat{A}[2] \setminus \emph{im} \, \phi_2.
\end{equation*}
\end{proposition}
\begin{proof}
Let us start by making a couple of observations. First, the trivial
line bundle $\mathcal{O}_A$ is obviously invariant for the monodromy
action. Second, for the computation of the monodromy we may assume
that $\textrm{NS}(A)$ is $1$-dimensional, generated by the numerical
class of $\mL$. Then for any $M \in G_{\Delta}$ the monodromy
transformation associated with $M$ sends $\mathcal{L}$ to
$\mathcal{L} \otimes \mathcal{Q}$, with $\mathcal{Q} \in
\widehat{A}$. Since $K(\mL)=K(\mL \otimes \mathcal{Q})$, it follows
that $\ker \phi_2$ is invariant under the monodromy action; hence
$\textrm{im}\, \phi_2^{\times}$ is invariant too. It remains to show
that $\widehat{A}[2] \setminus \textrm{im} \, \phi_2$ forms a single
orbit.

Set $A=\mathbb{C}^2/ \Lambda$ and write the period matrix for $A$ as
\begin{equation*}
\left(
  \begin{array}{cccc}
    z_{11} & z_{12} & 1 & 0 \\
    z_{21} & z_{22} & 0 & 2 \\
  \end{array}
\right),
\end{equation*}
with $Z:=\left(
\begin{array}{cc}
z_{11} & z_{12} \\
z_{21} & z_{22} \\
\end{array}
\right) \in \mathfrak{H}_2$. Then the lattice $\Lambda$ is spanned
by the four column vectors
\begin{equation*}
\lambda_1:=\left(
            \begin{array}{c}
              z_{11} \\
              z_{21} \\
            \end{array}
          \right), \quad
\lambda_2:=\left(
            \begin{array}{c}
              z_{12} \\
              z_{22} \\
            \end{array}
          \right), \quad
\mu_1:=\left(
            \begin{array}{c}
              1 \\
              0 \\
            \end{array}
          \right), \quad
\mu_2:=\left(
            \begin{array}{c}
              0 \\
              2 \\
            \end{array}
          \right)
\end{equation*}
and the matrix of the alternating form $E \colon \Lambda \times
\Lambda \to \mathbb{Z}$ with respect to this basis is $\left(
\begin{array}{cc}
0 & \Delta \\
-\Delta & 0 \\
\end{array}
\right).$ Therefore
\begin{equation} \label{eq.lm}
E(\lambda_1, \, \mu_1)=1, \quad E(\mu_1, \, \lambda_1)=-1, \quad
E(\lambda_2, \, \mu_2)=2, \quad E(\mu_2, \, \lambda_2)=-2
\end{equation}
and all the other values are $0$.

The finite subgroup $\widehat{A}[2]$ of $\widehat{A}$ is isomorphic
to $(\mathbb{Z}/2 \mathbb{Z})^4$ and, by the Appell-Humbert theorem,
its elements can be canonically identified with the $16$ characters
$\Lambda \to \mathbb{C}^*$ with values in $\{\pm 1\}$, see
\cite[Chapter 2]{BL04}. Since
\begin{equation*}
K(L)=\{x \in A \; | \, E(x, \, \Lambda) \subseteq \mathbb{Z} \},
\end{equation*}
it follows $K(L)= \langle \frac{\lambda_2}{2}, \, \frac{\mu_2}{2}
\rangle$ and $\textrm{im} \, \phi_2 = \langle
\phi_2(\frac{\lambda_1}{2}), \, \phi_2(\frac{\mu_1}{2}) \rangle$. In
other words, $\textrm{im}\, \phi_2$ corresponds to the four
characters
\begin{equation*}
e^{2 \pi i (\cdot, \, x)} \colon \Lambda \longrightarrow \{\pm 1\}
\end{equation*}
with $x=0, \, \frac{\lambda_1}{2}, \, \frac{\mu_1}{2}, \,
\frac{\lambda_1+\mu_1}{2}$. We will denote a character $\chi \colon
\Lambda \to \{\pm 1\}$ by the vector \\$(\chi(\lambda_1), \,
\chi(\lambda_2), \, \chi(\mu_1), \, \chi(\mu_2))$. Therefore
$\textrm{im} \, \phi_2$ consists of
\begin{equation*}
\chi_0:=(1,\,1,\,1,\,1), \quad \chi_1:=(1, \, 1, \, -1, \, 1), \quad
\chi_2:=(-1, \, 1, \, 1, \, 1), \quad \chi_3:=(-1, \, 1, \, -1, \,
1),
\end{equation*}
whereas the $12$ elements of $\widehat{A}[2] \setminus \textrm{im}
\, \phi_2$ correspond to
\begin{equation*}
\begin{array}{ccc}
 \psi_1 :=(1, \, 1, \, 1, \, -1),  &  \psi_2:=(1, \, 1, \, -1,
\,-1), & \psi_3:=(1, \, -1,\, 1, \, 1), \\
\psi_4 :=(1, \, -1, \, 1, \, -1),   &  \psi_5 :=(1, \, -1, \, -1, \, 1), & \psi_6:=(1, \, -1, \, -1, \,-1), \\
  \psi_7 :=(-1, \, 1,\, 1, \, -1), & \psi_8:=(-1, \, 1, \, -1, \, -1), &  \psi_9 :=(-1, \, -1, \, 1, \, 1),\\
  \psi_{10}:=(-1, \, -1, \, 1, \,-1), & \psi_{11}:=(-1, \, -1,\, -1, \, 1),
& \psi_{12} :=(-1, \, -1, \, -1, \, -1).
\end{array}
\end{equation*}
Now take $M=\left(
              \begin{array}{cc}
                \alpha & \beta\\
                \gamma & \delta \\
              \end{array}
            \right) \in G_{\Delta}$, where
\begin{equation*}
\alpha=\left(
  \begin{array}{cc}
    a_{11} & a_{12} \\
    2a_{21} & a_{22} \\
  \end{array}
\right), \quad \beta=\left(
  \begin{array}{cc}
    b_{11} & 2b_{12} \\
    2b_{21} & 2b_{22} \\
  \end{array}
\right), \quad \gamma=\left(
  \begin{array}{cc}
    c_{11} & c_{12} \\
    c_{21} & \frac{c_{22}}{2} \\
  \end{array}
\right), \quad \delta=\left(
  \begin{array}{cc}
    d_{11} & 2d_{12} \\
    d_{21} & d_{22} \\
  \end{array}
\right)
\end{equation*}
and $a_{ij}, \, b_{ij}, \, c_{ij}, d_{ij} \in \mathbb{Z}$. By
\cite[proof of Proposition 8.1.3]{BL04}, the monodromy action of $M$
on $\Lambda$ is given by the matrix $\left(
  \begin{array}{cc}
    \mathbb{I}_2 & 0 \\
    0 & \Delta \\
  \end{array}
\right)^{-1} \, {}^tM \,
 \left( \begin{array}{cc}
    \mathbb{I}_2 & 0 \\
    0 & \Delta \\
  \end{array}
\right)$, so the induced action over a character $\chi$ is as
follows:
\begin{equation} \label{eq.action.lattice}
\begin{split}
(M \cdot \chi)(\lambda_1) & = \chi(\lambda_1)^{a_{11}}
\chi(\lambda_2)^{a_{12}} \chi(\mu_1)^{b_{11}} \chi(\mu_2)^{b_{12}}, \\
(M \cdot \chi)(\lambda_2) & =\chi(\lambda_1)^{2a_{21}}
\chi(\lambda_2)^{a_{22}}
\chi(\mu_1)^{2b_{21}}\chi(\mu_2)^{b_{22}}, \\
(M \cdot \chi)(\mu_1) & = \chi(\lambda_1)^{c_{11}}
\chi(\lambda_2)^{c_{12}} \chi(\mu_1)^{d_{11}} \chi(\mu_2)^{d_{12}}, \\
(M \cdot \chi)(\mu_2) & =\chi(\lambda_1)^{2c_{21}}
\chi(\lambda_2)^{c_{22}} \chi(\mu_1)^{2d_{21}} \chi(\mu_2)^{d_{22}}.
\end{split}
\end{equation}
For instance we have
\begin{equation*}
M \cdot \chi_1=((-1)^{b_{11}}, \, 1, \, (-1)^{d_{11}}, \, 1), \quad
M \cdot \chi_2=((-1)^{a_{11}}, \, 1, \, (-1)^{c_{11}}, \, 1),
\end{equation*}
hence the set $\textrm{im} \, \phi_2^{\times}$ is
$G_{\Delta}$-invariant (and by using the matrices of type
$M_{\boldsymbol{b}}$ one checks that it is a single
$G_{\Delta}$-orbit, as expected).

Now we are ready to compute the monodromy action of $G_{\Delta}$ on
$\widehat{A}[2]\setminus \textrm{im}\, \phi_2$ or, equivalently, on
the set $\{\psi_1, \ldots, \psi_{12} \}$. By using
\eqref{eq.action.lattice}, one shows that
\begin{itemize}
\item the monodromy permutation associated with a matrix of type
$M_{\boldsymbol{b}}$ is
\begin{itemize}
\item[-] $(\psi_2 \, \psi_8)(\psi_5 \, \psi_{11})(\psi_6 \,
\psi_{12})$ if $b_{11}$ is odd and $b_{12}$, $b_{22}$ are even;
\item[-] $(\psi_1 \, \psi_7)(\psi_2 \, \psi_8)(\psi_4 \,
\psi_{10}) (\psi_{6} \, \psi_{12})$  if $b_{12}$ is odd and
$b_{11}$, $b_{22}$ are even; \item[-] $(\psi_1 \, \psi_4)(\psi_2 \,
\psi_6)(\psi_7 \, \psi_{10}) (\psi_{8} \, \psi_{12})$ if $b_{22}$ is
odd and $b_{11}$, $b_{12}$ are even;
\end{itemize}
\item the monodromy permutation associated with a matrix of type
$M_{\boldsymbol{d}}$ is
\begin{itemize}
\item[-] $(\psi_3 \, \psi_9)(\psi_4 \, \psi_{10})(\psi_5 \,
\psi_{11}) (\psi_6 \, \psi_{12})$ if $d_{21}$ is odd and $d_{12}$ is
even; \item[-] $(\psi_1 \, \psi_2)(\psi_4 \, \psi_6)(\psi_7 \,
\psi_8) (\psi_{10} \, \psi_{12})$ if $d_{12}$ is odd and $d_{21}$ is
even;
\end{itemize}
\item the monodromy permutation associated with the matrix
$M_{\boldsymbol{1, \, 2}}$ is

$(\psi_1 \, \psi_3)(\psi_2 \, \psi_9)(\psi_5 \, \psi_7) (\psi_6 \,
\psi_{10})(\psi_8 \, \psi_{11})$.
\end{itemize}

Therefore, the subgroup of the symmetric group $\mathcal{S}_{12}$
corresponding to the monodromy action of $G_{\Delta}$ on $\{\psi_1,
\ldots, \psi_{12} \}$ contains
\begin{equation*}
\begin{split}
T:=\langle & (2 \;8)(5 \; 11)(6 \; 12), \; (1 \;7)(2 \;8)(4 \;
10)(6 \; 12), \; (1 \;4)(2 \;6)(7 \; 10)(8 \; 12), \\
&(3 \;9)(4 \;10)(5 \; 11)(6 \; 12), \; (1 \; 2)(4 \;6)(7 \; 8)(10 \;
12), \; (1\;3)(2\;9)(5 \; 7)(6 \; 10)(8 \; 11) \rangle.
\end{split}
\end{equation*}
A straightforward computation, for instance by using the Computer
Algebra System \verb|GAP4| (see \cite{GAP4}), shows that $T$ is a
\emph{transitive} subgroup of $\mathcal{S}_{12}$; therefore
$\{\psi_1, \ldots, \psi_{12} \}$ form a single orbit for the
$G_{\Delta}$-action. This completes the proof.
\end{proof}

Now let $(A=\mathbb{C}^2/\Lambda, \, H)$ be a polarized abelian
surface of type $\Delta$. A symplectic basis  $\{\lambda_1,
\lambda_2, \mu_1, \mu_2 \}$ of $\Lambda$ for $H$ determines the $15$
non-trivial characters $\chi_1, \ldots, \chi_3$, $\psi_1, \ldots,
\psi_{12}$. Therefore we can consider the set of pairs
\begin{equation*}
(Z, \, \rho ), \quad Z \in \mathfrak{H}_2, \, \rho \in \{\chi_1,
\ldots, \chi_3, \, \psi_1, \ldots, \psi_{12} \} \subset
\widehat{A_Z}[2],
\end{equation*}
which can be seen as a subscheme of the relative Picard scheme
$\textrm{Pic}^0(\mathcal{X}_{\Delta}/ \mathfrak{H}_2)$.

The group $G_{\Delta}$ acts on this set of pairs, the action being
defined by \eqref{eq.action on H2} on the first component and by the
monodromy
 on the second one. The corresponding quotient
$\mathcal{A}_{\Delta}[2]$  is a quasi-projective variety and by
construction we have a degree $15$ cover $\pi \colon
\mathcal{A}_{\Delta}[2] \to \mathcal{A}_{\Delta}$. We can identify
$\mathcal{A}_{\Delta}[2]$ with the set of pairs $(A, \,
\mathcal{Q})$, where $A$ is the isomorphism class of a
$(1,\,2)$-polarized abelian variety and $\mathcal{Q}$ is a
non-trivial, $2$-torsion line bundle on $A$; then the map $\pi$ is
just the forgetful map $(A, \, \mathcal{Q}) \to A$.

\begin{proposition} \label{prop.components}
$\mathcal{A}_{\Delta}[2]$ is the disjoint union of two connected
components:
\begin{equation*}
\mathcal{A}_{\Delta}^{(a)}[2] \quad  \textrm{and} \quad
\mathcal{A}_{\Delta}^{(b)}[2],
\end{equation*}
corresponding to $\mathcal{Q} \notin \emph{im}\, \phi_2^{\times}$
and $\mathcal{Q} \in \emph{im}\, \phi_2^{\times}$, respectively.
The forgetful maps
\begin{equation*}
\pi_1 \colon \mathcal{A}_{\Delta}^{(a)}[2] \longrightarrow
\mathcal{A}_{\Delta}, \quad \pi_2 \colon
\mathcal{A}_{\Delta}^{(b)}[2]\longrightarrow \mathcal{A}_{\Delta}
\end{equation*}
are finite covers of degree $12$ and $3$. Finally, both
$\mathcal{A}_{\Delta}^{(a)}[2]$ and
$\mathcal{A}_{\Delta}^{(b)}[2]$ are irreducible and generically
smooth.
\end{proposition}
\begin{proof}
The first part of the statement follows immediately since the action
of $G_{\Delta}$ on the set of non-trivial characters $\Lambda \to
\{\pm 1\}$ has precisely two orbits, namely $\{\chi_1, \ldots \chi_3
\}$ and $\{\psi_1, \ldots, \psi_{12} \}$  (Proposition
\ref{prop.monodromy}). Moreover $\pi_1$ and $\pi_2$ are
$\acute{\textrm{e}}$tale covers on a smooth Zariski open set
$\mathcal{A}_{\Delta}^0 \subset \mathcal{A}_{\Delta}$; then they are
generically smooth. Finally, by construction
$\mathcal{A}_{\Delta}^{(a)}[2]$ and $\mathcal{A}_{\Delta}^{(b)}[2]$
are normal varieties, because they only have quotient singularities.
Then, since they are connected, they must be also irreducible.
\end{proof}

Similarly, there is an action of $G_{ \Delta}$ on the set of
triplets
\begin{equation*}
(Z, \, \chi, \, \chi^{1/2}),
\end{equation*}
where $Z \in \mathfrak{H}_2$, $\chi \in \{\chi_1, \, \chi_2, \chi_3
\} \subset \widehat{A_Z}[2]$ and $\chi^{1/2} \colon \Lambda_Z \to
\mathbb{C}^*$ is a character whose square is $\chi$. The
corresponding quotient is a quasi-projective variety that can be
identified with the space $\mathcal{A}_{\Delta}[2, \, 4]$ of triples
$(A, \, \mQ, \, \mQ^{1/2})$, where $A$ is the isomorphism class of a
$(1, \, 2)$-polarized abelian surface, $\mQ \in \textrm{im}\,
\phi_2$ and $\mQ^{1/2}$ is a square root of $\mQ$. There is
forgetful map $\pi \colon \mathcal{A}_{\Delta}[2, \,4] \to
\mathcal{A}_{\Delta}$, sending $(A, \, \mQ, \, \mQ^{1/2})$ to $A$;
it is a finite cover of degree $48$.

\begin{proposition} \label{prop.monodromy.2}
$\mathcal{A}_{\Delta}[2, \,4]$ is irreducible and generically
smooth.
\end{proposition}
\begin{proof}
It is sufficient to check that the monodromy action of
$G_{\Delta}$ is transitive on the set
\begin{equation*}
\{(\mQ, \, \mQ^{1/2}) \, | \, \mQ \in \textrm{im}\,
\phi_2^{\times}, \, (\mQ^{1/2})^2 = \mQ \}.
\end{equation*}
This is a straightforward computation which can be carried out as
the one in the proof of Proposition \ref{prop.monodromy}, so it is
left to the reader.
\end{proof}


\bigskip
\bigskip

Matteo Penegini, Lehrstuhl Mathematik VIII, Universit\"at
Bayreuth, NWII, D-95440 Bayreuth, \\ Germany \\ \emph{E-mail
address:}
 \verb|matteo.penegini@uni-bayreuth.de| \\ \\

Francesco Polizzi, Dipartimento di Matematica, Universit\`{a}
della
Calabria, Cubo 30B, 87036 \\ Arcavacata di Rende (Cosenza), Italy\\
\emph{E-mail address:} \verb|polizzi@mat.unical.it|

\end{document}